\documentclass[10pt,twoside]{article}
\usepackage{amssymb,latexsym,amsmath,epsfig,amsthm} 

\makeatletter

\renewcommand\section{\@startsection {section}{1}{\z@}
{-30pt \@plus -1ex \@minus -.2ex}
{2.3ex \@plus.2ex}
{\normalfont\normalsize\bfseries\boldmath}}

\renewcommand\subsection{\@startsection{subsection}{2}{\z@}
{-3.25ex\@plus -1ex \@minus -.2ex}
{1.5ex \@plus .2ex}
{\normalfont\normalsize\bfseries\boldmath}}

\renewcommand{\@seccntformat}[1]{\csname the#1\endcsname. }

\makeatother

\theoremstyle{definition}

\usepackage{amsmath,amsthm}
\usepackage{amssymb}
\usepackage{latexsym}
\usepackage{enumitem}
\usepackage{graphicx}

\usepackage{xcolor}

\DeclareOldFontCommand{\ttoriginal}{\normalfont\ttfamily}{\mathtt}

\newcommand{\tttt}{{{{{''}}}}}
\renewcommand{\tt}{{''}}

\newcommand{\taux}{\gamma}

\usepackage{hyperref}

\usepackage{relsize}

\font\smallit=cmti10

\newtheorem{thm}{Theorem}[section]
\newtheorem{cor}[thm]{Corollary}

\newtheorem{defn}[thm]{Definition}

\newtheorem{rem}[thm]{Remark}
\newtheorem{conj}{Conjecture}[section]
\newtheorem{example}{Example}[section]

\newcommand{\N}{\mathbb N}
\newcommand{\Zp}{\mathbb Z^+}
\newcommand{\Z}{\mathbb Z}
\newcommand{\Q}{\mathbb Q}

\newcommand{\uu}{u}
\renewcommand{\int}{\bigcap}

\newcommand{\eq}{equation}

\newcommand{\weset}{We set $f(b,c)=\d_2(b)-\d_2(c)+\nu_2\left({2c-b-1 \choose c}\right)$.}

\newlength{\notewidth}
\renewcommand{\d}{\sigma}
\newcommand{\n}{h}

\theoremstyle{definition}

\numberwithin{equation}{section}

\usepackage[displaymath, mathlines]{lineno}
\begin{document}

\pagestyle{myheadings}
\markboth{New results on the $p$-adic valuation of Stirling numbers}{A.~Adelberg and T.~Lengyel}
\thispagestyle{empty}
\baselineskip=12.875pt
\vskip 30pt

\begin{center}
\uppercase{\bf {New results} on the $p$-adic valuation of Stirling numbers}
\vskip 20pt
{\bf Arnold Adelberg}\\
{\smallit Department of Mathematics and Statistics,
Grinnell College,
Grinnell, IA~50112, USA }\\
{\ttoriginal adelbe@math.grinnell.edu}\\
\vskip 10pt
{\bf Tam\'as Lengyel}\\
{\smallit Mathematics Department,
Occidental College,
Los Angeles, CA~90041, USA }\\
{\ttoriginal  lengyel@oxy.edu}\\
\end{center}
\vskip 20pt

\vskip 30pt

\centerline{\bf Abstract}

\noindent
We generalize results on the $p$-adic {valuations} of $S(n,k)$, the Stirling number of the second kind and $s(n,k)$, the
Stirling number of the first kind.
We have several new estimates for these valuations, along with criteria for when the estimates are sharp.
The primary foci are the explicit evaluation of $\nu_2(S(n,k))$ with $n=c2^\n, k=b 2^\n+a$,
$a,b,c,\n,k,n\in \Zp$, and $1\le a\le 2^{\n-1}$, and $\nu_p(S(n, k))$ when $n=c p^\n$ for an odd prime $p$.
{We have strong new results, which generalize and strengthen previous results, for all primes.}
We also have some new  results on the $p$-adic valuations $\nu_p(s(n,k))$ for all primes.
We generally assume that $p-1\mid n-k$ for exact values of $\nu_p(S(n,k))$ or $\nu_p(s(n,k))$.
In addition, we have proved some new Amdeberhan-type identities for Stirling numbers of both kinds.
We also extend some recent results
and propose two new conjectures, as well as proofs and extensions of previous ones.\\

\section{Introduction}
Let $n,k  \in \Zp$,
and $\d_p(k)$ and $\nu_p(k)$ denote the digit sum
in the base $p$ representation of $k$ and the highest power of the prime $p$ dividing $k$, respectively.
In particular for $p=2$,
 let $[n]$ be the set of $2$-powers in the base $2$ expansion  of $n$, i.e., $[n]$ corresponds to the ones in the expansion, and $\sigma_2(n)=\#([n])$.
In order to make the paper self-contained as much as possible, more
definitions are provided in Section~\ref{sec:appendix}.
\\

In this paper we study certain $p$-adic properties of the Stirling numbers of the second kind $S(n,k)$ and of the first kind
$s(n,k)$.
In recent years several papers appeared in which the $p$-adic properties of these numbers and the $2$-adic properties, in particular have been investigated.
{Various efforts were made to determine congruences and the $p$-adic valuations of both kinds of Stirling numbers, e.g., \cite{Adelberg}, \cite{Adseq}, \cite{AMM}, \cite{ChanManna}, \cite{Davis}, {\cite{GL}}, \cite{HZZ},   \cite{tlFQ}, \cite{DMTCS}, \cite{Len10}, \cite{p-adic}, \cite{PM}, \cite{ZZH1}, \cite{ZZH2}, \cite{SdW}, and \cite{QHsp}, \cite{KY}, \cite{Lens},  \cite{LS},  \cite{QHs}, etc.}
Some of these results are universal and deal with $\nu_p(S(n,k))$ and $\nu_p(s(n,k))$ with no particular restrictions on $n$ and $k$ while other results focus on cases when $n$ is in the form of $c p^\n$ with $c,\n \in \Zp$. There are result of an experimental nature too, e.g., Amdeberhan et al. \cite{AMM} establish patterns in the 2-adic evaluation of $S(n,k)$ with different values of $k$.\\

The first efforts to analyze the $p$-adic behaviour of $S(n,k)$ focused on cases
with $n=cp^\n$ and $\n$ sufficiently large with or without specific lower bounds on $\n$.
A special $2$-adic study of the Stirling numbers of the second kind was
initiated by Lengyel in 1994 (cf. \cite[Theorems~1 and 2]{tlFQ}),  which established that $\nu_2(S(c 2^\n,k))=\d_2(k)-1$ for all
$\n\ge k-2$ and contained a conjecture that was proven by De Wannemacker in
\cite{SdW} in 2005:
\begin{thm}[{\cite[Theorem~1]{SdW}}]
\label{th:SdW}
Let $\n,k\in \N$ and $1\le k\le 2^\n$. Then we have
\begin{\eq}
\label{eq:SdW}
\nu_2(S(2^\n,k))=\d_2(k)-1.
\end{\eq}
\end{thm}
Lengyel extended this theorem in
\begin{thm}[{\cite[Theorem~2]{DMTCS}}]
\label{th:tlc}
Let $c,k,\n \in \N$ and $1\le k\le 2^\n$, then
\begin{\eq}
\label{eq:tlc}
\nu_2(S(c 2^\n,k))=\d_2(k)-1.
\end{\eq}
\end{thm}
He also provided alternative proofs in \cite{Len10}.
The major development came in a sequence of papers by Adelberg (cf. \cite{Adelberg} and \cite{Adseq}) that depended on his previous work on higher order Bernoulli numbers and polynomials,
and on the relation between higher order Bernoulli numbers and Stirling polynomials.
In \cite[Theorem~2.2]{Adelberg}, Adelberg  generalized \eqref{eq:SdW} for
arbitrary primes and replaced the term $2^\n$ with the condition that $S(n,k)$
is a so-called minimum zero case (MZC).
Note that $S(c2^\n,k), 1\le k\le 2^\n$ with $c$ odd is a MZC if $c=1$, but not otherwise.
A remarkable short and instructive proof of Theorem~\ref{th:SdW} is presented as a special case in \cite[Theorem~2.2]{Adelberg}.
To deal with the case of $c>1$ and odd, Adelberg also introduced the notion of almost
minimum zero case (AMZC) in \cite{Adseq} that resulted in a short proof of Theorem~\ref{th:tlc}.
We define these concepts in relation with inequality \eqref{eq:DeW}.
De Wannemacker's useful inequality \cite[Theorem~3]{SdW} states that
\begin{\eq}
\label{eq:DeW}
\nu_2(S(n,k))\ge \d_2(k)-\d_2(n)
\end{\eq}
with $n,k \in \mathbb N$ and $0\le k\le n$. \\

Although \eqref{eq:DeW} is general in $n$ and $k$, it was Adelberg who developed a general theory in \cite{Adelberg} and \cite{Adseq} that removed the limitations due to the assumed nature of $n=c 2^\n$ and $k: \, 1\le k\le 2^\n$ in $S(n,k)$ that were used in Theorems~\ref{th:SdW} and \ref{th:tlc}.\\

Adelberg
called the estimate \eqref{eq:DeW} the minimum zero estimate, which he
improved
in several ways. He generalized it to odd primes, and proved a criterion for when
it is sharp, i.e., when $S(n,k)$ is a minimum zero case (MZC).
  He also added an additional term to improve the estimate in all cases where it is not sharp, namely he proved in \cite[Theorem~3.3]{Adseq} that
\begin{\eq}
\label{eq:AA}
\nu_2(S(n,k))\ge \d_2(k)-\d_2(n)+N,
\end{\eq}
where $N=\#\{[n]
\bigcap [n-k]\}$, i.e., the number of $2$-powers common to the base 2
expansions of $n$ and $n-k$,
which he  called the almost minimum zero estimate.  Adelberg showed that the MZC occurs if and only if  $N=0$ in \eqref{eq:AA}, i.e., $\nu_2(S(n,k))=\d_2(k)-\d_2(n)$.
%
He also showed that his new estimate is always nonnegative, unlike De Wannemacker's estimate (1.3), which can be negative.  If \eqref{eq:DeW} is not sharp but \eqref{eq:AA} is sharp, we have an almost minimum zero case (AMZC).
In addition, he obtained certain ``shifted" estimates and cases for $p=2$, which obviate the need for many inductions.
In this paper we generalize these concepts to arbitrary primes. In particular, the shifted cases give us insight into certain cases where $\nu_p(S(n+1,k+1))=\nu_p(S(n,k))$ and where $\nu_p(s(n-1,k-1))
=\nu_p(s(n,k))$, continuing the generalization that was done in \cite{Adelberg} of what was originally a conjecture of Amdeberhan et al.; cf. \cite{AMM}. {We give geometric interpretations of all the cases in terms of the Newton polygons of certain higher order Bernoulli polynomials.}

 The cases when we get exact values for the $p$-adic valuations generally assume that $p-1\mid n-k$.
 Our analysis is based on the ``maximum poles" of higher order Bernoulli polynomials, which are the highest powers of $p$ in the denominators of their coefficients.
Using these polynomials is a unique feature of the presented method, which generalizes the approach used in \cite{Adseq} for $p=2$.
The results for $p=2$ are simpler and more explicit {than} the results for odd primes, primarily because the maximum pole is much easier to determine for $p=2$, where there is a simple formula depending on base 2 expansions.

In the analysis of $\nu_2(S(c2^\n,k))$ the range $k\in [1,2^{\n+1}]$ has been analyzed
in \cite[Theorem~3.4]{Adseq} with $c\ge 3$ odd. In this paper we focus on different
ranges and sets of $k$ values. Of course, Theorem~\ref{th:tlc} has already dealt
with the range $k\in [1,2^\n]$ and, as it turns out, the answer essentially relies
on the fact that $S(c2^\n,k)$ is an AMZC. The  upper range $k \in [(c-1)2^\n+1,c 2^\n]$
is covered by a theorem due to Zhao et al.; cf. \cite{ZZH1}:

\begin{thm} [{\cite[Theorem~1.2]{ZZH1}}]
\label{th:c-1}
 Let $a, c, \n \in \N$  with $c \ge 1$ being odd and
$1 \le a \le 2^\n$.
Then 

\begin{\eq}
\label{eq:ZZH}
\nu_2(S(c2^\n, (c-1)2^\n + a)) = \d_2(a)-1.
\end{\eq}
\end{thm}

Adelberg provided a short proof in \cite[Theorem~2.1]{Adseq}  by observing
that in this case $S(c2^\n, (c-1)2^\n + a)$ is a MZC.\\

We focus our study on the remaining range $k\in [2^\n+1,(c-1)2^\n]$
and derive one of our main results Theorem~\ref{th:GC} on $\nu_2(S(n,k))$ with 
$n=c2^h, c>1, k=b 2^h+a, 1\le a \le 2^{h-1}$. 
The special case when $b=1$, which was a strong form of a long-standing conjecture of Lengyel, was proven by Adelberg in
\cite[Theorem~2.7]{Adseq}. 
 {The proof illustrates Adelberg's method, which depends on the analysis of higher order Bernoulli numbers and polynomials.}
 It is further extended in
Theorem~\ref{th:+1+2}, with an Amdeberhan-type result.
\\

An important result for the calculation of $p$-adic valuations $\nu_p(S(n,k))$ for odd primes $p$ is Theorem~\ref{th:Snkp2}, which broadly generalizes  previous {known}  results.
We give further theorems of the same type with increasing generality, and which illustrate what is essential for the calculations; cf. Theorems~\ref{th:Arnieplus} and \ref{th:Arniex}.
{Theorem~\ref{th:Snkp2}
 and its first corollary, Corollary~\ref{cor:Snkp2cor},  show that if ${p^h\mid n}$, $k\le p^h$, and ${p-1\mid n-k}$ then $\nu_p(S(n,k))=\lfloor\sigma_p(k)-1)/(p-1)\rfloor$.  This result is consistent with the formula proven in \cite[Theorem~1]{GL} (cf. Theorem~\ref{th:GL}),
but without the restrictive assumption that ${p-1\mid n}$.}
Since $\nu_p(k!) =(k-\sigma_p(k))/(p-1)$ {in {the formula in} \cite{GL} for $\nu_p(k!S(n,k))$,} the estimates and exact values given
are equivalent to ours. 
Our theorem also gets 
the estimate without that restrictive assumption 
${p-1\mid n-k}$, which is in fact the almost minimum zero estimate for $\nu_p(S(n,k))$. The remarkable thing about the exact value and estimate is that they are independent of $n$ and $h$. Corollary~\ref{cor:Snkp2cor} generalizes \eqref{eq:tlc}
to all primes.
\\

 Our primary tool for obtaining these results is the analysis of higher order Bernoulli polynomials, and specifically of their poles, which is determined by consideration of their Newton polygons.  The Appendix collects the essential background, and includes the new results Theorem~\ref{th:5.x}
 for odd primes and Theorem~\ref{th:5.x1-5.11}
 for $p=2$, which lead directly to the theorems about Stirling numbers. These Appendix theorems have great importance for this paper.
 \\

We also extend the Amdeberhan-type results in Theorems~\ref{th:3.x}, \ref{th:newAmdeMZC},
 \ref{th:Amde}, \ref{th:3.??}, and \ref{th:3.21} for $S(n,k)$, as well as results in Sections~\ref{sec:main2} and \ref{sec:appendix}; cf. Theorems~\ref{th:4.?}, \ref{th:4.?2}, \ref{th:4.???} and \ref{th:4.???1} for $s(n,k)$. \\

To a large extent Adelberg reduces the study of $p$-adic valuation of Stirling numbers to the study of $p$-adic valuation of binomial coefficients; cf. \cite {Adelberg} and \cite{Adseq}. We extend his techniques  in the current paper, which provide a unified treatment for Stirling numbers of both kinds and illustrate the duality between Stirling numbers of the first and second kinds.\\

\newcommand{\two}{two }

Section~\ref{sec:prep} summarizes some useful results, including criteria for the different cases. 
{Besides
  Theorem~\ref{th:GC}, noted above,
  Section~\ref{sec:main1} contains
  Theorems~\ref{th:3.x} and \ref{th:3.?}, which are very useful for determining the cases,  i.e., when our estimates are sharp, for odd primes and for the even prime.}
  {These theorems provide the mechanism to establish AMZ and SAMZ cases and Amdeberhan-type results.}
  Examples are provided to illustrate some of the results.
  It also contains several major theorems that vastly generalize previous results on these $p$-adic values, as well as \two new conjectures, Conjectures~\ref{conj:stat} and
  \ref{con:a+2}.
  In Theorem~\ref{th:p-adic}
  we prove a result which
  corrects and sharpens
  another open conjecture of Lengyel (cf.
  \cite[Conjecture~1]{p-adic}),
  and which leads to a consideration of Fibbinary numbers. We
also state Theorem~\ref{th:p-adic2}, which is an extension of a previously proven result Theorem~\ref{th:TH26}(a), with a similar proof.
We note that in 2014 Davis \cite[Theorems~1.1, 1.2, and Propositions~3.1]{Davis} proved similar results (cf. \cite{ChanManna}, too), apparently not knowing about
\cite[Conjecture~1]{p-adic} and \cite[Theorem~5]{DMTCS}
and without  providing  a good estimate for when the limiting $2$-adic value is achieved.
%
Section~\ref{sec:main2} deals with Stirling numbers of the first kind, and contains many results analogous to those in the earlier sections.
Our goal is to demonstrate the applicability of our methods, but the results regarding Stirling numbers of the first kind are somewhat limited in scope.
 An Appendix is included as Section~\ref{sec:appendix} in which we provide the requisite definitions and results, detailed descriptions of the techniques applied,
 as well as some essential new material including the fundamental Theorems~\ref{th:5.x} and \ref{th:5.x1-5.11}.\\


\section{Preparation}
\label{sec:prep}
\subsection{Summary of previous results}
\label{sec:prepprior}

\noindent
For $p=2$  the most general result, Theorem~\ref{th:GC}, deals with the range $k=b2^\n+a \in [b2^\n+1,b2^\n+2^{\n-1}]$ with $b\in\N$,
first we study the $k$ values for which $\nu_2(k)\ge \n$.
Adelberg
found the following version of \cite[Theorem~5]{DMTCS} which originally proved that
$\nu_2(S(c2^\n,b2^\n))=\d_2(c-b)+\d_2(b-1)-\d_2(2c-b-1)$
for $b,c \in \N, 1\le b\le c$ and $\n$ sufficiently large.
Adelberg's version \cite[Theorem~2.6]{Adseq} provides a lower bound on $\n$ which is
exponentially better than the one in \cite{DMTCS} and we will use it in
Conjecture~\ref{conj:stat}.\\

\begin{thm}[{\cite[Theorem~2.6]{Adseq}}]
\label{th:TH26}
With $b, c \in \Zp$ and $b\le c$,
\begin{itemize}
\item[(a)] if $\nu_2(b) < \nu_2(c)$ or if $\nu_2(b)=\nu_2(c)$ and
$2^{\nu_2(c-b)}\in [b]$, then
\[\lim_{\n\to\infty} \nu_2(S(c2^\n, b2^\n)) = \d_2(b)-\d_2(c) + \nu_2\left({2c-b\choose c}\right)\]
and this limit is attained if $2^{\n-1+\nu_2(c-b)}\ge \nu_2\left({2c-b\choose c}\right)$;
\\
%
%
\item [(b)] if $\nu_2(c) < \nu_2(b)$ or $\nu_2(c) =
\nu_2(b)$ and $2^{\nu_2(c-b)}\in [c]$, then
\[\lim_{\n\to\infty} \nu_2(S(c2^\n, b2^\n)) = \d_2(b-1)-\d_2(c-1) + \nu_2\left({2c-b-1\choose c-1}\right).\]
%
%
Furthermore, if $\nu_2(c) < \nu_2(b)$, the limit is attained if $2^{\n-1+\nu_2(c)} \ge \nu_2\left({2c-b-1
\choose c-1}\right)$,
while if $\nu_2(c)=\nu_2(b)$ and $2^{\nu_2(c-b)}\in [c]$, the limit is attained if $2^{\n-1+\nu_2(c)} > \nu_2\left({2c-b-1\choose c-1}\right)$.
\end{itemize}
\end{thm}

\begin{rem}
\label{rem:notsharp}
\hskip.05in

\begin{itemize}
\item[(i)] Part (a) of the preceding theorem is characterized by $2^{\nu_2(c-b)} \in [b]$ and part (b){,} by $2^{\nu_2(c-b)} \in [c]$.

\item[(ii)] If $\nu_2(b)\le \nu_2(c)$ then $\nu_2({2c-b\choose c})=\nu_2(b)-\nu_2(c) +\nu_2({2c-b-1\choose c-1})$, implying that the limit formula given in part (a) agrees with the formula in part (b) in this case, and the estimate for when the limit is attained is at least as good.

\item[(iii)] The estimates given in Theorem~\ref{th:TH26} are not always sharp.
For example, if $c=30$ and $b=1$, then
$\nu_2(S( c 2^h, 2^h))=\d_2(2^h)-1=0$, by Theorem~\ref{th:tlc} for all $h\ge 0$. On the other hand,  the estimate of Theorem~\ref{th:TH26}(a) requires only that $2^{h-1}\ge \nu_2({59\choose 30})=3$, i.e., that $h\ge 3$.\\
\end{itemize}
\end{rem}


In the range $k\in [2^\n+1,2^\n+2^{\n-1}]$ we consider
\begin{\eq}
\label{eq:genc}
\nu_2(S(c 2^\n,2^\n+a))
\end{\eq}
 with $c\ge 3$ and $1\le a \le 2^{\n-1}$.
 Unless $c\ge 2$ is an even integer, in which case the 2-adic order is simply $\d_2(a)$ by Theorem~\ref{th:tlc}, it turns out that, as opposed to the cases covered by Theorems~\ref{th:tlc} and \ref{th:c-1}, the 2-adic valuation depends on $\nu_2(c-1)$ as it was suggested in

\begin{conj}[{\cite[Conjecture~1]{DMTCS}}]
\label{conj:r}
We have
\begin{\eq}
\label{eq:conjr0}
\nu_2(S((2^r+1)2^\n, 2^\n+a))=\d_2(a)+r.
\end{\eq}
for integers $r\ge 1$, $1\le a \le 2^{\n-1}$, and sufficiently large $\n$.\\
\end{conj}

Lengyel also proved the following generalized version of the conjecture although based on the assumption that $\nu_2(a)$ and thus, $\n$ are large.

\begin{thm}[{\cite[Theorem~3]{DMTCS}}]
\label{th:c}
We have
\begin{\eq}
\label{eq:ccond}
\nu_2(S(c2^\n, 2^\n+a))=\d_2(a)+\nu_2(c-1)
\end{\eq}
for $c\ge 3$ odd, $1\le a <2^\n$, if
\begin{equation*}
\nu_2(a)-\d_2(a)>\nu_2(c-1)+1.
\end{equation*}
\end{thm}

Numerical evidence suggested that Conjecture~\ref{conj:r} and Theorem~\ref{th:c} hold true with much weaker assumptions. The papers by Adelberg (cf. \cite{Adelberg} and \cite{Adseq}) provide us with the technical tools to handle these cases.
Adelberg \cite{Adseq}  gives the general form in Theorem~\ref{th:GCpre}{,} which contains  Conjecture~\ref{conj:r} and Theorem~\ref{th:c} with improved explicit bounds on $\n${,}  as special cases.

\begin{thm}[{\cite[Theorem~2.7]{Adseq}}]
\label{th:GCpre}
Let $c \ge 3$ be odd and $1\le  a \le 2^{\n-1}$.
Then
\[\nu_2(S(c 2^\n, 2^\n+a)) = \sigma_2(a) + \nu_2(c-1)\]
if $2^{\n-2} \ge  \sigma_2(c) + \nu_2(c-1)-1=\sigma_2(c-2)+1$.
\end{thm}

Theorem~\ref{th:GCpre} implies that Conjecture~\ref{conj:r} is true as the special case with $\sigma_2(c)=2$, and it provides a bound on the smallest value of $h$.

\begin{cor}[{\cite[Corollary~5.2]{Adseq}}]
\label{th:conj}
We have
\begin{\eq}
\label{eq:conjr}
\nu_2(S((2^r+1)2^\n, 2^\n+a))=\d_2(a)+r
\end{\eq}
for integers $r\ge 1$, $1\le a \le 2^{\n-1}$, and $\n\ge \lceil \log_2 (r+1) \rceil +2$.
\end{cor}

In this paper we further generalize Theorem~\ref{th:GCpre}. In fact, one of our main results is Theorem~\ref{th:GC} in which we determine the exact value of $\nu_2(S(c2^\n, b2^\n+a))$ with $a,
c, \n \in \Zp$,
$b\in \N$, and $1\le a \le 2^{\n-1}$.
Finally, we propose Conjecture~\ref{conj:stat},
which deals with the
cases in a statistical sense.
We include related results on $\nu_2(S(c2^h+1, b2^h+a+1))$ and $\nu_2(S(c2^h+2, b 2^h+a))$ as Theorem~\ref{th:+1+2} and Conjecture~\ref{con:a+2}.\\

For odd primes $p$ the following theorems apply. We will generalize them in Theorem~\ref{th:Snkp2}.
\begin{thm}[{\cite[Theorem~1]{GL}}]
\label{th:GL}
For an odd prime $p$, if $n= c (p-1)p^\n$, $1\le k\le n$, $c$ and $\n$ are positive integers such that
$GCD(c,p)= 1$, $\n$  is sufficiently large, and $k/p$ is not an odd integer, then
\[\nu_p(k! S(c(p-1)p^\n,k))=
\left\lfloor \frac{k-1}{p-1}\right\rfloor+\taux_p(k)
\]
where $\taux_p(k)$ is a nonnegative integer. Moreover, if $k$ is a multiple of $p-1$, then $\taux_p(k)= 0$.
\end{thm}

For $p=2$ and $1\le k\le 2^h$ the result also holds without the requirement that $k/p$ is not an odd integer, according to Theorem~\ref{th:tlc}.\\

Note that Theorem~\ref{th:GL}, which assumes that ${p-1\mid n}$ 
and provides an estimate for $\nu_p(S(n,k))$, agrees with our almost minimum zero estimate by Theorem~\ref{th:Snkp2}. If ${p-1\mid k}$ as well, then the estimate is sharp, and we have an almost minimum zero case. The theorem also assumes that ${p^h \mid n}$, with $h$ sufficiently large. Observe also our Theorem~\ref{th:Snkp2} gives us a good bound for what it means for $h$ to be sufficiently large. 
Related results can be found in \cite{SunFleck2}.\\

Other aspects of $\nu_p(S(c(p-1)p^h,k))$ with $c\ge 1$ integer are discussed and explored using $p$-adic analytic techniques by Miska in \cite{PM}, which lead to the proof of conjectures suggested in \cite{AMM}.\\

\begin{thm}[{\cite[Theorem~2.2]{Adelberg}}]
\label{th:A2.2}
Let $n = cp^h$ with $1\le c\le p-1$ and assume that $1\le k\le n$ and
$p-1\mid n-k$. Then $S(n, k)$ is a MZC and
\[\nu_p(S(n, k)) =
\frac{\sigma_p(k)-\sigma_p(n)}{p-1}=\frac{\sigma_p(k)-c}{p-1}.\]
\end{thm}

\begin{rem}
\label{rem:A2.2}
Note that there is no special bound on $h$ in Theorems~\ref{th:SdW}, \ref{th:tlc}, \ref{th:c-1}, and \ref{th:A2.2}, while in Theorems~ \ref{th:TH26}, \ref{th:GCpre}, \ref{th:GC}, \ref{th:p-adic},
\ref{th:p-adic2}, \ref{th:Snkp2},
and Corollary~\ref{th:conj} effective lower bounds are given on $h$. Theorems~\ref{th:c}, \ref{th:GL},  and Conjecture~\ref{conj:r} require sufficiently large values of $h$ without specific bounds.
\end{rem}

\subsection{Estimates and cases for $\nu_p(S(n,k))$ for all primes $p$}
\label{sec:estimates}
In \cite{Adelberg} Adelberg considered the minimum zero estimate (MZ) 
\begin{\eq}
\label{eq:MZ}
\nu_p(S(n,k)\ge (\d_p(k)-\d_p(n))/(p-1).
\end{\eq}
  He found simple necessary and sufficient conditions for when this estimate is sharp, called the minimum zero case (MZC), namely
  $r=(n-k)/(p-1)\in \N$, and $S(n,k)$ is a MZC if and only if $p\nmid {n+r\choose r}$.
  He tied this result to higher order Bernoulli numbers by showing that $S(n,k)$ is a MZC if an only if $\nu_p(B_{n-k}^{(-k)})=-\d_p(n-k)/(p-1)$.
  In \cite{Adseq}, dealing only with the prime $p=2$, Adelberg found the additional shifted minimum zero estimate (SMZ) $\nu_2(S(n,k))\ge \d_2(k-1)-\d_2(n-1)$, and showed that this is sharp, called shifted minimum zero case (SMZC) if and only if $2\nmid {n-1+r\choose r}$. In addition, he also found more refined estimates depending on the base 2 expansions of $n$, $k$, and $n-k$, called the almost minimum zero (AMZ) and shifted almost minimum zero  (SAMZ) estimates, and defined the almost minimum zero case (AMZC) and shifted almost minimum zero case (SAMZC) as when these estimates are sharp.
  {He noted that these estimates
depended on
the maximum poles
of the associated higher order Bernoulli polynomials, i.e., the highest powers of 2 in the denominators of their coefficients. \\

{In this paper, we are interested in extending these concepts
to all primes $p$, so we have to define them for odd $p$ and establish basic properties. The key  formulas and definitions are summarized in Section~\ref{sec:appendix}. We have the connecting formulas (cf. \eqref{eq:S}) 
\begin{\eq}
\label{eq:new}
S(n,k)={n \choose k} B_{n-k}^{(-k)}
\end{\eq}
and
\begin{\eq}
\label{eq:newshift}
S(n,k)={n-1\choose k-1} B_{n-k}^{(-k+1)}(1).
\end{\eq}
The second formula, called shifted, is obtained by replacing $(n,k)$ by $(n-1,k-1)$  in the higher order Bernoulli polynomial $B_{n-k}^{(-k)}(x)$ , and then replacing the constant term $B_{n-k}^{(-k)}$  by $B_{n-k}^{(-k+1)} (1)$, which is the constant term of $B_{n-k}^{(-k+1)} (x+1)$.
From the second formula we get the shifted minimum zero estimate (SMZ):

\begin{\eq}
\label{eq:new2}
\nu_p(S(n,k))\ge (\d_p(k-1)-\d_p(n-1))/(p-1).
\end{\eq}
When this estimate is sharp we have the shifted minimum zero case (SMZC).\\

The SMZ estimate  can be easily expressed in terms of the MZ estimate using \
\eqref{eq:new2}, namely the SMZ estimate is $\nu_p(S(n,k))\ge (\d_p(k)-\d_p(n))/(p-1) + \nu_p(k)-\nu_p(n)$
by \eqref{eq:A-1},
so the two estimates are equal if and only if $\nu_p(k)=\nu_p(n)$ and the SMZ estimate is better (bigger) if and only if $\nu_p(n)< \nu_p(k)$. In particular, we get the following result.
\\

\begin{thm}
\label{th:newth}
If $S(n,k)$ is  a MZC then $\nu_p(k)\le \nu_p(n)$, while if $S(n,k)$ is a SMZC, then $\nu_p(n)\le \nu_p(k)$.
\end{thm}

\begin{rem}
\label{rem:firstAmde}
Obviously, the SMZ estimate for $S(n+1,k+1)$ is the same as the MZ estimate for $S(n,k)$. We will now establish simple criteria for the MZ and SMZ cases, and use that in Section~\ref{sec:Amde} to determine  conditions that insure $\nu_p(S(n+1,k+1))=\nu_p(S(n,k))$, which we call an ``Amdeberhan-type identity."
\end{rem}

The following theorem follows easily  from the discussion of poles summarized in Section~\ref{sec:5.3}. Our unified treatment via higher order Bernoulli numbers and polynomials allows us to consider both kinds of Stirling numbers (cf. Section~\ref{sec:main2}).

\begin{thm}
\label{th:2.9}
If {$l\in \Z$} the ${\nu_p(B_n^{(l)})=-\d_p(n)/(p-1)}$ if and only if $\nu_p(B_n^{(l)} (1))={-\d_p(n)/(p-1)}$.
\end{thm}
}

\begin{proof}[Proof of Theorem~\ref{th:2.9}]
 The condition that $\nu_p(B_n^{(l)})=-\d_p(n)/(p-1)$ is equivalent to $B_n^{(l)}$ has a pole of order $\d_p(n)/(p-1)$, so $B_n^{(l)}$, which is the constant coefficient of the higher order Bernoulli polynomial $B_n^{(l)} (x)$ has a greater pole (smaller $p$-adic value) than all other coefficients of the polynomial, so this is equivalent to $\nu_p(B_n^{(l)} (1))$, which is the sum of all coefficients, has the same pole $\d_p(n)/(p-1)$.                                                                  \end{proof}

The equation $\nu_p(B_n^{(l)})=-\d_p(n)/(p-1)$ is equivalent to $B_n^{(l)}(x)$ is a maximum pole case (MPC)  as defined in \cite{Fib98}, and the equation $\nu_p(B_n^{(l)} (1)) =-\d_p(n)/(p-1)$ is equivalent to $B_n^{(l-1)} (x)$ is a shifted maximum pole case (SMP) as defined in Definition~\ref{def:new}.

\begin{cor}
\label{cor:2.?}
$S(n,k)$ is a MZC if and only if $S(n+1,k+1)$ is a SMZC, and if $S(n,k)$ is a MZC  and/or  $S(n+1,k+1)$ is a SMZC, then $\nu_p(S(n,k))=\nu_p(S(n+1,k+1)) =(\d_p(k)-\d_p(n)/(p-1)$.
\end{cor}

\begin{proof}[Proof of Corollary~\ref{cor:2.?}]
 From \eqref{eq:new} and Legendre's Theorem, $S(n,k)$ is a MZC if and only if $\nu_p(B_{n-k}^{(-k)})=-\d_p(n-k)/(p-1)$, and similarly from \eqref{eq:newshift},
  $S(n+1,k+1)$ is a SMZC if and only if $\nu_p(B_{n-k}^{(-k)}) (1))=-\d_p(n-k)/(p-1)$.
Thus, the conditions that $S(n,k)$ is an MZC and $S(n+1,k+1)$ is a SMZC are equivalent.  If $S(n,k)$ is a MZC then $\nu_p(S(n,k))=(\d_p(k)-\d_p(n))/(p-1)$. Similarly, if $S(n+1,k+1)$ is a SMZC then $\nu_p(S(n+1,k+1)) =
(\d_p(k)-\d_p(n))/(p-1)$.  Hence if $S(n,k)$ is a MZC  and/or $S(n+1,k+1)$ is a SMZC, then $\nu_p(S(n,k))=\nu_p(S(n+1,k+1))$.                                    \end{proof}

\begin{rem}
\label{rem:remx}
This corollary greatly strengthens \cite[Theorem~2.5]{Adelberg}
which proved that if $S(n,k)$ is a MZC then $\nu_p(S(n,k))=\nu_p(S(n+1,k+1))$.  Since the concept of AMZC was unavailable at that time, that is all that could be proved, but we now have a
simpler and more instructive proof.
\end{rem}


The formulas \eqref{eq:new} and \eqref{eq:newshift} lead to two new estimates. Let $M$ be the highest order pole of $B_{n-k}^{(-k)} (x)$, i.e., the highest power of $p$ in the denominator of its coefficients, and let $M'$ be similarly defined for $B_{n-k}^{(-k+1)} (x)$, which is obtained by replacing $(n,k)$ by $(n-1,k-1)$.
Then, since $\nu_p(B_{n-k}^{(-k)}) \ge -M$ and $\nu_p(B_{n-k}^{(-k+1)} (1))\ge -M'$, we get from \eqref{eq:new} and \eqref{eq:newshift}
the estimates

\begin{linenomath}
\begin{align}
\label{eq:new3}
\!\!\!\!\!\!\nu_p(S(n,k))&\ge \nu_p\left({n\choose k}\right) -M\notag\\
&=(\d_p(k)-\d_p(n))/(p-1) +\d_p(n-k)/(p-1) -M
\end{align}
\end{linenomath}
and
\begin{linenomath}
\begin{align}
\label{eq:new4}
\!\!\!\!\!\!\nu_p(S(n,k))&\ge \nu_p\left({n-1\choose k-1}\right)-M'\notag\\
&=(\d_p(k-1)-\d_p(n-1))/(p-1)+\d_p(n-k)/ (p-1) -M'.
\end{align}
\end{linenomath}
}

Since $M$ and $M'$ are less than or equal to $\d_p(n-k)/(p-1)$, we see that these
new
estimates are at least as good as the MZ and SMZ estimates, and are strictly  better if we do not have a MZC or SMZC, respectively.

The first estimate is called the almost minimum zero estimate (AMZ) and if it is sharp we have the AMZC (almost minimum zero case).  We sometimes
distinguish the AMZC  from the MZC by requiring that it not be a MZC, i.e., that
$\nu_p(S(n,k))\neq (\d_p(k)-\d_p(n))/(p-1)$.
Similarly, the second estimate is the shifted minimum zero estimate (SAMZ) and when it is sharp we have the shifted almost minimum zero case (SAMC), which we may stipulate is not a SMZC.
\\

Obviously, the SAMZ estimate for $S(n+1,k+1)$ is the same as the AMZ estimate for $S(n,k)$.\\

 The four cases for Stirling numbers of the second kind correspond to the four cases for higher order Bernoulli polynomials defined in Section~\ref{sec:SsB}, and the estimates for Stirling numbers of the second kind correspond to the estimates that are implicit for the higher order Bernoulli numbers.  Since there are similar connecting formulas with Stirling numbers of the first kind, we will develop similar estimates and cases for the Stirling numbers of the first kind.\\

For Stirling numbers of the second kind with $p=2$, we know that $M=\#([n-k]-[n])$,
where  $A-B$ denotes the set difference if $A$ and $B$ are sets, and $M'=\#([n-k]-[n-1])$, which agrees with the definitions given in \cite{Adseq}.  We do not have a simple formula for the maximum pole for odd primes $p$, but we do have an algorithm, which involves the Kimura $N$-function, for computing  $M$.  This algorithm was first established in \cite{Acta} and is described in Section~\ref{sec:appendix}; cf. Definition~\ref{def:Kimura}. We also know that $M$ is the maximum pole of $\{\tau_\uu | w(\uu)\le n-k\}$, and similarly
$M'$ is the maximum pole of $\{\tau_\uu | w(\uu)= n-k\}$, with $l$ replaced by $l+1$, using notations of the Appendix; cf. \cite{Acta} and Section~\ref{sec:appendix}.

We will show
if $p$ is odd and $p-1\mid n-k$
that these estimates for odd primes $p$ are non-vacuous (cf. Theorem~\ref{th:2.x}), and use the material in the Appendix to show there are only two possible candidates for  $\nu_p((n-k)! t_\uu) =-M$, with $w\le n-k$ and only one for $\nu_p((n-k)! t_\uu) =-M'$, with $w=n-k$, thereby establishing criteria for the AMZ and SAMZ cases, and finding simple formulas for $\nu_p(S(n,k))$ in these cases (cf. Theorem~\ref{th:3.x}).

\begin{thm}
\label{th:2.x}  The AMZ and SAMZ estimates for Stirling numbers of the second kind are non-negative.
\end{thm}

\begin{proof}[Proof of Theorem~\ref{th:2.x}]
 We prove this for the AMZ estimate of $S(n,k)$.  The proof for the SAMZ estimate is similar. The justifications are all in the Appendix, primarily in Section~\ref{sec:pary}.

If $M$ is the maximum pole of $B_{n-k}^{(-k)} (x)$, {then} the (longest) Kimura chain $N_1, \dots, N_M$ has length $M$, and $n+N_i/(p-1)$ has no base $p$ carries for each $i$. If $i$ is arbitrary and
$e=\nu_p(N_i)$,  and $c$ is the coefficient of $p^e$ in the base $p$ expansion of $N_i$, then $p-c$ is the coefficient of $p^e$ in $N_i/(p-1)$. If $n_e$ is the coefficient of $p^e$ in $n$, then $n_e +p-c<p$ since there is no base $p$ carry.  Thus $n_e<c$ for each $i$.
Hence by Legendre's Theorem, we have
$\nu_p({n\choose n-k})\ge M$, i.e., the AMZ estimate $\nu_p({n\choose k})-M$ is non-negative.
Clearly this implies that the AMZ estimate of $\nu_p(S(n,k)$ is between 0 and $\nu_p({n\choose k})$, and similarly that the SAMZ estimate of $S(n,k)$ is between 0 and $\nu_p({n-1\choose k-1})$.
\end{proof}

\begin{cor}
\label{cor:2.15}
 If $\nu_p(S(n,k))=0$ then $S(n,k)$ is
 a MZC or an AMZC and also a SMZC or a SAMZC.\end{cor}

\section{Main results for $S(n,k)$}
\label{sec:main1}
\subsection{The exact $2$-adic order}
\label{sec:exact}
One of the original goals of this paper was to prove the vastly improved generalization of Conjecture~\ref{conj:r},
Theorem~\ref{th:GC}, by using generating functions as in \cite{DMTCS} but succeeded only in the case of $a=1$. Induction proof did not seem to work either.
Although Conjecture~\ref{conj:r} is the most special case of Theorem~\ref{th:GC} with $\d_2(c)=2$ and $c$ odd,
not even this case was proven until Adelberg's
technique, relying on the three canonical partitions (cf. \cite{Adseq}) came to the rescue
in his proof of Theorem~\ref{th:GCpre}; cf. \cite[Theorem~2.7]{Adseq}. 
{The proof of this theorem is very similar to the special case $b=1$ given previously.}\\

\begin{thm}
\label{th:GC}
Let $a,c,\n \in \Zp$, $b\in \N$, and $n=c 2^\n$ where $c > 1$.
Assume that $1\le  a \le 2^{\n-1}$, and $k=b2^\n +a < n$.
We set $f(b,c)=\d_2(b)-\d_2(c)+\nu_2\left({2c-b-1 \choose c}\right)$.
%
We have
\begin{\eq}
\label{eq:GC2}
\nu_2(S(n,k))
=\d_2(k)-\d_2(n)+\nu_2\left({n-k+n \choose n}\right)
=\d_2(a)+f(b,c)
\end{\eq}
 if $2^{\n-2} \ge \nu_2\left({2c-b-1 \choose c}\right)$. If $b$ is odd{,} then {the right hand side of} \eqref{eq:GC2} {equals}
 \begin{\eq}
 \label{eq:GC2bodd}
 \d_2(a)+\nu_2(c-b)+
\nu_2\left({2c-b\choose b}\right),
\end{\eq}
since then $f(b,c)=\nu_2(c-b)+\nu_2\left({2c-b\choose b}\right)$.
\end{thm}

\begin{rem}
\label{rem:a=1}
We note that
if we also assume that $c>b\ge 1$ are odd integers
then  $2(c-b)$ and $b$ have no common base 2 digits for $b=1$ and $3$;
thus, the last term in
\eqref{eq:GC2bodd} is 0 {in these cases,}  as in Theorem~\ref{th:GCpre}.
\end{rem}


\begin{proof}[Proof of Theorem~\ref{th:GC}]
Assume that
$n=c2^\n$ and $k=b2^\n +a$, with $1\le a\le  2^{\n-1}$,
and $k \le n$, so $b<c$. We will prove that
$\nu_2(S(n,k))=\d_2(k)-\d_2(n) +\nu_2\left({n-k+n \choose n}\right)=\d_2(b)+\d_2(a)-\d_2(c) +\nu_2\left({n-k+n \choose n}\right) =\d_2(b)+\d_2(a)-\d_2(c) +\nu_2\left({c-b-1+c \choose c}\right)$ if $2^{\n-2}\ge \nu_2\left({c-b-1+c \choose c}\right)$.

Observe that
$n-k=(c-b)2^\n -a=(c-b-1)2^\n +2^{\n-1} +(2^{\n-1} -a)$, so $\nu_2\left({n-k+n \choose n}\right)=\nu_2\left({c-b-1)2^\n +c2^\n \choose c2^\n}\right)=\nu_2\left({c-b-1+c \choose c}\right)$.\\

We use terminology from \cite{Adseq}; see Section~\ref{sec:appendix} {also}.\\

Let $\uu$ be a partition with $w \le n-k$.  Then
$n-k=T+B$ where $T$ is the sum of 2-powers with exponent
at least $\n$,
i.e., $T=(c-b-1)2^\n$, and
$B=2^{\n-1} +(2^{\n-1} -a)$. \\

We have two cases. If $d=d(\uu)\ge T$, then $[d]$ contains all the 2-powers of $n-k$ with exponent
at least $\n$, so $\nu_2\left({n-k+n \choose n}\right)=\nu_2\left({T+n \choose n}\right)$, and it follows that $\nu_2(2^{n-k}t_\uu) \ge \nu_2\left({n-k+k \choose  n}\right)$, with equality if and only if $\nu_2(\uu)=n-k$, i.e.,  exactly if $u_1=n-k$.  That takes care of all the terms with big $d$, i.e., $d \ge T$.

In the other case $d<T$ and then $n-k-d>B=2^{\n-1} +(2^{\n-1} -a) \ge 2^{\n-1}$.  Then $n-k-\nu_2(\uu) \ge (n-k-d)/2 \ge 2^{\n-2}\ge \nu_2\left({n-k+n \choose n}\right)$ if $2^{\n-2} \ge \nu_2\left({n-k+n \choose n}\right)$, so for these terms as well $\nu_2(2^{n-k}t_\uu) > \nu_2\left({n-k+n \choose n}\right)$.

Thus the partition with $u_1 = n-k$ gives the unique dominant term (lowest value), and the theorem is proved.\\

Finally, we prove that the two expressions in
\eqref{eq:GC2} and \eqref{eq:GC2bodd}
are equal provided that $b$ is odd.
\\

We use identity \eqref{eq:A-1}.  Let $b$ be odd.
 Then $\d_2(a) +\d_2(b) - \d_2(c) +\nu_2({2c-b-1\choose c}) =\d_2(a)+\d_2(b)-\d_2(c) +
\d_2(c)+\d_2(c-b-1)-\d_2(2 c-b-1)=\d_2(a)+\d_2(b)+
(\d_2(c-b) -1+\nu_2(c-b)) -(\d_2(2 c-b)-1+\nu_2(2 c-b)) =\d_2(a)+\nu_2(c-b)+\d_2(b)+\d_2(2(c-b))-\d_2(2 c-b)=
\d_2(a)+\nu_2(c-b)+\nu_2({2c-b \choose b})$, {since $2c-b$ is odd, so $\nu_2(2c-b) =0$.}

{We note that if $b$ is odd and is a bottom segment of $c$ then $2c-b=2(c-b)+b$ and 
$[2(c-b)]\bigcap [b]=\emptyset$, so the last binomial coefficient is odd. Thus, we recover the formula $\nu_2(S(c 2^h,b 2^h +a) =\sigma_2(a) +\nu_2(c-b)=\sigma_2(a) +\alpha$, where $2^\alpha =$ smallest $2$-power in $[c]$ greater  than all the $2$-powers in $[b]$.}
\end{proof}

We prove the following theorem, which corrects and sharpens an open conjecture made in 2012, \cite[Conjecture~1]{p-adic}).  We will subsequently extend this result in Theorem~\ref{th:p-adic2}.\\
\begin{thm}
\label{th:p-adic}
For $c,\n,L \in \N$
if $c$ odd and $\n$ is sufficiently large, we have that $\nu_2(S(c2^{\n+1}+L,c 2^\n+L))$ is constant for all $0\le L<2^\n$. In fact, we have
\begin{\eq}
\label{eq:p-adic}
\nu_2(S(c2^{\n+1}+L,c 2^\n+L))=2\sigma_2(c)-\sigma_2(3c).
\end{\eq}
The lower bound $2^{\n-1}\ge \nu_2({3c\choose c})$ for $\n$ suffices.
\end{thm}
\begin{proof}[Proof of Theorem~\ref{th:p-adic}]
 We use the notations and results of Section~\ref{sec:appendix}, especially of Section~\ref{sec:criteria}.  In particular, $[a]$ is the set of $2$-powers in the base 2 expansion of $a$, and if $\uu$ is a partition then $t_\uu ={s\choose d} {d\choose \uu} /\Lambda^{\uu}$ where $s=-k-(n-k)-1=-n-1$, so that ${s\choose d}=(-1)^d {d+n\choose n}$.
We know by formula \eqref{eq:S}
that $S(n,k)={n\choose k} B_{n-k}^{(-k)}$, and
$B_{n-k}^{(-k)} =(-1)^{n-k} (n-k)! \sum_{w\le n-k} t_\uu$ by formula \eqref{eq:partition}.
    If $n=c2^{h+1} +L$ and $k=c2^h +L$, then $n-k =c2^h$ and $\nu_2(n-k) =h$. Let $\uu'$ be the partition  concentrated  in place 1 with $u'_1 = n-k$. The idea of the proof is to show that if $\uu\neq \uu'$ and $w(\uu)\le n-k$ then $\nu_2(t_\uu)>\nu_2(t_{\uu'})$.  Hence $\nu_2(B_{n-k}^{(-k)})=\nu_2((n-k)! t_{\uu'})=
    -\d_2(n-k) +\nu_2({n-k+n\choose n})= -\d_2(c)+\nu_2({c2^h +c2^{h+1}+L\choose c2^h})=-\d_2(c) +\nu_2({3c\choose c}) =\d_2(c)-\d_2(3c)$ if $L<2^h$, which we are assuming.\\

There are two cases. If $d=d(\uu)$, the first case is $d\ge n-k-2^{\nu_2(n-k) } =c2^h -2^h=(c-1) 2^h$, and the second case is when $n-k-d>2^h$. Since $2^h \in [n-k]$
and $2^h>L$, it follows that $\nu_2({n-k-2^h +n\choose n})=\nu_2({n-k+n\choose n})$, and if $n-k-2^h\le d\le n-k$ then $
   \nu_2({d+n \choose n})\ge \nu_2({n-k+n\choose n})$.  But $w(\uu)\le n-k$, so $\nu_2(u)=\nu_2 (\Lambda^\uu) \le n-k$ with equality if and only if $\uu=\uu'$ (cf. \cite[Lemma~2]{JNT}).
   Hence for this case $\nu_2 (t_\uu) >\nu_2(t_{\uu'})$ if $\uu\neq \uu'$.\\

   Next consider the other case, where $n-k-d>2^h$. By
   \cite[Corollary~A.1]{Adseq},
   we have $n-k-\nu_2(\uu) \ge (n-k-d)/2>2^{h-1}$, so if $2^{h-1}\ge \nu_2 ({3c\choose c})$
   then
   $n-k-\nu_2 (\uu)> \nu_2({n-k+n\choose n-k})=\nu_2({3c\choose c})$.

    Hence the term when $\uu=\uu'$ dominates the sum (i.e., has smallest $2$-adic value), and $\nu_2(B_{n-k}^{(-k)})=\nu_2((n-k)! t_{\uu'})=-\d_2(c) +\nu_2({3c\choose c})$, which implies by \eqref{eq:new} (cf. \eqref{eq:S})
    that
    $\nu_2 (S(n,k)) =\nu_2({n \choose k}) +\nu_2(B_{n-k}^{(-k)})=\d_2(k) -\d_2(n) +\d_2(n-k)-\d_2(c) +\nu_2({3c \choose c})  =\d_2(c)+\d_2(L) -(\d_2(c) +\d_2(L))+
    2\d_2(c)  -\d_2(3 c) =2\d_2(c) -\d_2(3 c)$ (if $2^h >L$ and $2^{h-1} \ge \nu_2({3 c\choose c})=2 \d_2(c) -\d_2(3 c)$).
\end{proof}
Theorem~\ref{th:p-adic} in fact can be generalized to the following extension of Theorem~\ref{th:TH26}(a), whose proof is entirely similar, and will be omitted.
%
\begin{thm}
\label{th:p-adic2}
Suppose that the hypotheses of Theorem~\ref{th:TH26}(a) are satisfied and $0\le L<2^h$.  Then
 the conclusions hold if $S(c 2^h,b 2^h)$ is replaced by $S(c 2^h +L, b 2^h +L)$.  We have that
\begin{\eq}
\label{eq:p-adic1}
\lim_{h\to\infty} \nu_2(S(c 2^h+L,b 2^h +L))
\end{\eq}
 exists and is independent of $L$, and is attained if
$2^{h-1  +\nu_2(c-b)} \ge
\nu_2({2 c-b \choose c})$.
 \end{thm}
\begin{rem}
\label{rem:p-adic}
The original conjecture claimed that $\nu_2(S(c2^{h+1}+L,c2^h+L))=\nu_2({3c\choose c})=2\d_2(c)-\d_2(3c)$ with $c$ odd, $0\le L\le 2^h$. It suggested that $h\ge 3$ might suffice but in fact, we need that $h$ is sufficiently large.
The case $L=2^h$ follows from Theorem~\ref{th:TH26}(b),
namely $n=c2^{h+1} +2^h=(2c+1)2^h$ and $k=c2^h+2^h=(c+1)2^h$. Since $c$ is odd, $2c+1$ is odd and $c+1$ is even, so by Theorem~\ref{th:TH26}(b), we have $\nu_2(S(n,k))=\d_2(c)-\d_2(2c)+
\nu_2({4c+2-c-2\choose c})=\nu_2({3c\choose c})$ if $2^{h-1}>\nu_2({3c\choose c})$
 Note that the $2$-adic order is zero if and only if $c$ is Fibbinary.
\end{rem}

In 2014 Davis \cite[Theorems~1.1,~1.2, and Proposition~3.1]{Davis} proved results similar to Theorems~\ref{th:p-adic} and  \ref{th:p-adic2} in greater generality since they hold for any odd prime $p$ when $c\equiv b \bmod {(p-1)}$ and allow different additive terms in \eqref{eq:p-adic1}.
However,
he did not provide a good estimate for when the limiting value is achieved.\\

\subsection{The exact $p$-adic order}
\label{sec:exactp}
In this section we generalize some of the results for
odd primes $p$, although they may apply to $p=2$ as well.
Much of the terminology is explained in Section~\ref{sec:appendix}, and many of the proofs follow directly from the results in that section.

\begin{thm}
\label{th:3.x}
Let $p$ be an odd prime.
Assume that $\nu_p(k)\le \nu_p(n)$ and $r=(n-k)/(p-1)\in \N$. Then the following are equivalent
\begin{itemize}
\item[(1)] $S(n,k)$ is an AMZC;
\item[(2)] $S(n+1,k+1)$ is a SAMZC;
\item[(3)] $\nu_p({r+n\choose n})=\d_p(n-k)/(p-1) -M$, where $M$ is the  maximum pole of
$B_{n-k}^{(-k)} (x)$;
\item[(4)] $\nu_p(\tau_\uu)=-M$, where $u_{p-1}=r
$; cf. Section~\ref{sec:5.3}.
\end{itemize}
Furthermore, if these conditions hold then
$\nu_p(S(n+1,k+1))=\nu_p(S(n,k))=(\d_p(k)-\d_p(n))/(p-1) +\nu_p({r+n\choose n})$.
\end{thm}

The proof follows immediately from
Theorem~\ref{th:5.x}, where the partition $\uu'$ concentrated in place $p-1$ with $u'_{p-1} =r-1$ is eliminated from consideration by the assumption that $\nu_p(k)\le \nu_p(n)$. If $\nu_p(k)>\nu_p(n)$, then $\nu_p(k)>0$ and $\nu_p(k-1)=0$, so we can replace $(n,k)$ by $(n-1,k-1)$ and $M$ by $M'$, the maximum pole of $B_{n-k}^{-(k-1)} (x)$ in
Theorem~\ref{th:3.x}.\\

For $p=2$ there are small modifications: There are now 3 canonical partitions (cf. Theorem~\ref{th:5.x1-5.11}), but the one with $u_1 =n-k-1$ has weight less than $n-k$, so is not involved in $B_{n-k}^{(-k)} (1)$, and is also ruled out by the hypothesis that $\nu_2(k)\le \nu_2(n)$. We then have conditions (1) and (2) are equivalent, and if they hold, then the Amdeberhan equality
(cf. Remark~\ref{rem:firstAmde})
holds. The other analysis can be changed to reflect that there are now two possible  candidate partitions $u$ and $u'$ with $u_1=n-k$ and with $u'_1=n-k-3$ and $u'_3 =1$,
exactly one of which must work for an AMZC.

\begin{rem}
\label{rem:31.4}
 The preceding Theorem~\ref{th:3.x}  contains some very general results that apply to our cases.  In particular, the last formula gives a very simple value for $\nu_p(S(n,k))$, which does not depend on the maximum pole of any higher order Bernoulli polynomials, which may be difficult to calculate.  Of course, the AMZC and SAMZC implicitly involve the maximum poles, but the binomial coefficient encodes that, and demonstrates precisely how we have improved the MZ estimate. The situation is similar for the SMZ estimate.
\end{rem}

\begin{proof}[Proof of Theorem~\ref{th:3.x}]
We use Theorem~\ref{th:5.x}, with $n:=n-k$ and $l:=-k$.
Then, if $M$ is the maximum pole of $B_{n-k}^{(-k)} (x)$, the
assumption that $\nu_p(k)\le \nu_p(n)$ translates to $\nu_p(l)\le
\nu_p(n)$, so the partition $\uu$ with $u_{p-1} =r$ is the only
possible
candidate for $\nu_p((n-k)! t_\uu) =-M$, among all partitions $\uu$
with $w(\uu)\le n-k$.
Thus, $\nu_p(B_{n-k}^{(-k)})=-M$ if and only if
$\nu_p(\tau_\uu)=-M$, if and only if $\nu_p(B_{n-k}^{(-k)} (1))=-M$.
The first of these equalities defines $S(n,k)$ is a AMZC, and the
last equality defines $S(n+1,k+1)$ is a SAMZC. Since
$\tau_\uu=(n-k)! {l-(n-k)-1\choose d}/p^r$, we have
$\nu_p(\tau_\uu)=(n-k-\d_p(n-k))/(p-1) +\nu_p({n+r\choose n})-r=
-\d_p(n-k)/(p-1) +\nu_p({n+r\choose n})$, so the middle equality
says that $\nu_p(\tau_\uu) = -M$, i.e., that is the unique
candidate for the maximum pole.
       Thus, conditions (1)-(4) are equivalent. If we now assume that they hold, then $\nu_p(B_{n-k}^{(-k)}) =\nu_p(B_{n-k}^{(-(k+1)+1)} (1)) =-M$, and $\nu_p(S(n,k))=(\d_p(k)-\d_p(n))/(p-1) + \d_p(n-k)/(p-1) -M =(\d_p(k)-\d_p(n))/(p-1)+\nu_p({n+r\choose n})$.
\end{proof}

The nice thing about this formula for $\nu_p(S(n,k))$ is that it
does not involve any a priori value of the maximum pole. Of course,
it is only valid if we have the conditions (1)--(4), e.g., if $S(n,k)$
is an AMZC, in addition to $p-1\mid n-k$ and $\nu_p(k)\le
\nu_p(n)$.

We have a similar formula for $p=2$ based on the definitions and
analysis in \cite{Adseq}. The situation is slightly more
complicated, because for $p=2$ there are three canonical
partitions, and one of which is
ruled out by the assumption that
$\nu_2(k)\le \nu_2(n)$.

\begin{thm}
\label{th:3.?}
For $p=2$, assume that $\nu_2(k)\le \nu_2(n)$.
Then the
following are equivalent
\begin{itemize}
\item[(1)] $S(n,k)$ is an AMZC;
\item[(2)] $S(n+1,k+1)$
is a SAMZC;
\item[(3)] $\nu_2({n+n-k\choose n})=\d_2(n-k) -M$ or
$\nu_2({n+n-k-2\choose n})=
\d_2(n-k) -(M+1)$ and $n-k$ is odd, where the or is exclusive and
$M$ is the maximum pole of $B_{n-k}^{(-k)}(x)$;
\item[(4)] If $u_1 =n-k$ and $u'_1=n-k-3$, $u'_3 =1$ with $n-k$ odd, then
$\nu_2(\tau_\uu)=-M$ or $\nu_2(\tau_{\uu'})=-M$, again with exclusive
or.
\end{itemize}
Furthermore if these conditions hold then
$\nu_2(S(n,k))=\nu_2(S(n+1,k+1))$. If $\nu_2(\tau_\uu)=-M$, then
$\nu_2(S(n,k))=\d_2(k) -\d_2(n) +\nu_2({n+n-k\choose n})$. If
$\nu_2(\tau_{\uu'}) =-M$ then
$\nu_2(S(n,k))=\d_2(k)-\d_2(n)+\nu_2({n+n-k-2\choose n}) - 1$.
\end{thm}

The proof of this theorem follows the same pattern as the previous
one, with the additional complication that even after eliminating
the partition $u_1=n-k-1$ by the assumption that $\nu_2(k)\le
\nu_2(n)$, we are still left with two partitions, exactly one of
which giving the maximum pole, and both having weight $n-k$. The
one that works determines $\nu_2(S(n,k))$. The $2$-adic values
differ by one.
\\

\begin{rem}
\label{rem:31.4.2}
For any prime $p$, if $\nu_p(k)>\nu_p(n)$ then $\nu_p(k-1)=0\le \nu_p(n-1)$, so we can apply Theorems~\ref{th:3.x} and \ref{th:3.?} with $(n,k)$ replaced by $(n-1,k-1)$.
\end{rem}

The main theorems in this section follow from Theorem~\ref{th:3.x}.
Our primary task is showing that the hypotheses of Theorem~\ref{th:3.x}
are satisfied. {We define the least positive residue $k'$ {of $k \bmod{(p-1)}$} by $1\le k'\le p-1$ {and} $k\equiv k' \bmod {(p-1)}$.}
\\

\begin{thm}
\label{th:Snkp2}
Let $n=c p^\n$ where $c=\sum_{i\ge 0} c_i p^i$ is the base $p$ expansion of $c$. Let $k'$ be the least positive residue of $k \bmod {(p-1)}$.
 Assume that $0<k\le \min\{c_0,k'\}  p^\n$.
 \begin{itemize}
 \item[(a)] If $\d_p(n)\ge k'$ and $n'$ is the smallest segment of {(the base $p$ expansion of)} $n$ with $\d_p(n')=k'$, then the maximum pole $B_{n-k}^{(-k)}(x)$ equals $\d_p(n'-k)/(p-1)$
      and the almost minimum zero estimate is $\nu_p(S(n,k))\ge $
      $(\d_p(k)-k')/(p-1)$.
 \item[(b)] If $n\equiv k \bmod {(p-1)}$  then $S(n,k)$ is an AMZC
and $\nu_p(S(n,k))=(\d_p(k)-k')/(p-1)$.
 \end{itemize}
\end{thm}

Note that the estimate depends only on $k$ and is independent of $n$.\\

\begin{cor}
\label{cor:Snkp2cor}
Suppose that $p^h\mid n$ and {$0<k\le p^h$}. If $p-1\mid n-k$ then $S(n,k)$ is an AMZC and
$\nu_p(S(n,k))=(\sigma_p(k) -k')/(p-1)= \lfloor(\sigma_p(k)-1)/(p-1)\rfloor${, using the floor function.}
\end{cor}

\begin{rem}
\label{rem:1010}
Theorem~\ref{th:Snkp2}
 generalizes the case with $1\le c\le p-1$,
which is a MZC as it is given in Theorem~\ref{th:A2.2}; cf. \cite[Theorem~2.2]{Adelberg}. It also generalizes previous results, e.g., the $p=2$ case;
cf. Theorem~\ref{th:tlc} and \cite[Theorem~3.4 (i) and Thereom~3.5 (b)]{Adseq}, which is an AMZC.
Apparently, the first result for odd primes was given in Theorem~\ref{th:GL} in 2001; cf. \cite[Theorem~1]{GL},
 which assumes that $n$ rather than $n-k$ is a multiple of $p-1$, and also that $k/p$ is not an odd integer, and gives an exact value
 when $n$ and $n-k$ are divisible by $p-1$.
 Theorem~\ref{th:Snkp2} agrees with the result, moreover
its part (b) provides the exact value of $\nu_p(S(n,k))$ for all values of $h\ge
\log_p (k/c_0)$
if $p-1|n-k$.
If $p-1\mid n$ then $\d_p(n)\ge k'$, so Theorem~\ref{th:Snkp2} fully generalizes and extends Theorem~\ref{th:GL} for all primes. Note also that if $n\equiv k \bmod {(p-1)}$ then $\d_p(k)\equiv \d_p(n) \bmod {(p-1)}$, so that $\d_p(n)\ge k'$, i.e., that hypothesis of part (b) implies the hypothesis of part (a).
\end{rem}

\begin{proof}[Proof of Theorem~\ref{th:Snkp2}]
Follows easily from the material in maximum poles and the Kimura $N$-function in the Appendix, and from Theorem~\ref{th:3.x}.
First prove part (a).  {Note that since $n'$ is the bottom segment of the base $p$ representation of $n$ with $\sigma_p(n') =k'$, we see that  $n'$ is a generalization of the Kimura $N$-function, namely if
{$p-1\mid k$} then $n'=N(n;p)$.}\\

Let $N_1=N(n'-k;p), N_2=N(n'-k-N_1;p), \dots, N_M=N(n'-k-N_1-\dots-N_{M-1};p)$.
Then $\sum_{1\le i \le M} N_i =n'-k$, and $\d_p(N_i)=p-1 $ and $\sum \d_p(N_i)$ has no base $p$ carries. This gives $M (p-1)=\d_p(n'-k)$, i.e. $M=\d_p(n'-k)/(p-1)$. If $i<M$, all $p$-powers in $N_i$ are less than  $p^h$, so $n+N_i/(p-1)$ has no base $p$ carries.  On the other hand, since $\d_p(n')\le p-1$, it follows that $\nu_p(N_M)<h$. But $N_M +N_M/(p-1)$ has only a single base $p$ carry, occurring in the lowest $p$-power of $N_M$.  Hence $n+N_M/(p-1)$ has no base $p$ carry, and thus
$n+(n'-k)/(p-1)$ has no base $p$ carry.\\

On the other hand, if $N$ is a segment of $n-k-(n'-k)=n-n'$, then $N+N/(p-1)$ has a base $p$ carry.  But $N$ is a segment of $n$, so $n+N/(p-1$) has a base $p$ carry.  Thus $N_1,N_2,\dots,N_M$ is the (longest) Kimura chain, and $M$ is the maximum pole of $B_{n-k}^{(-k)} (x)$.\\

The AMZ estimate is $\nu_p(S(n,k))\ge (\d_p(k)-\d_p(n))/(p-1) +
(\d_p(n-k)-\d_p(n'-k))/(p-1) =(\d_p(k)-\d_p(n)+\d_p(n-n'))/(p-1) =(\d_p(k)-\d_p(n')/(p-1)=(\d_p(k)-k')/(p-1)$.\\

For part (b),  since $k\le c_0 p^h$, we have $\nu_p(k)\le \nu_p(n) =h$.  To finish the proof, by Theorem~\ref{th:3.x},
we have only to show that  $\nu_p({n+r\choose r})=\d_p(n-k)/(p-1) -M=(\d_p(n-k)-\d_p(n'-k))/(p-1) =\d_p(n-n')/(p-1) =(\d_p(n)-k')/(p-1)$.\\

Since $n+r=n+(n-k)/(p-1) =n+(n-n')/(p-1) +(n'-k)/(p-1)$, and $n+(n'-k)/(p-1)$ has no base $p$ carries, it suffices to consider $n+(n-n')/(p-1) =n-n'+(n-n')/(p-1) +n'$.  Since the number of base $p$ carries for $n-n'+(n-n')/(p-1)$ is $\d_p(n-n')/(p-1)$, it follows that the number of carries is $\d_p(n-n')/(p-1) =(\d_p(n)-k')/(p-1)=\nu_p({n+r\choose n})$, so $S(n,k)$ is an AMZC and the estimate in (a) is sharp.
\end{proof}


Here are some easy consequences of Theorem~\ref{th:Snkp2} and some general examples.

\begin{cor}
\label{cor:spec}
Suppose that $S(n,k)$ satisfies
all the hypotheses of Theorem~\ref{th:Snkp2}.
Then $p\mid S(n,k)$ if and only if $\d_p(k)>p-1$.
\end{cor}
\begin{cor}
\label{cor:spec2}Let $0<k\le p-1$ and suppose that $p\mid n$ and $p-1\mid n-k$.  Then $p\nmid S(n,k)$.
\end{cor}

\begin{proof}[Proof of Corollary~\ref{cor:spec}]
We have that $\nu_p(S(n,k)) =(\d_p(k) -k')/(p-1)$
where $k'$ is the least positive residue of $k$, i.e., the least positive residue of $\d_p(k)$, and $1\le k' \le p-1$. Observe that $\nu_p(S(n,k))=0$ if and only if $\d_p(k)=k'$ and $1\le k'
\le p-1$.
\end{proof}

\begin{proof}[Proof of Corollary~\ref{cor:spec2}]
It is easy to see that the hypotheses of the Theorem~\ref{th:Snkp2} are satisfied.
Since $k'=k=\d_p(k)$, the result is immediate.
\end{proof}

\begin{example}
\label{ex:example}
We illustrate the use of Theorem~\ref{th:Snkp2} to calculate $\nu_p(S(n,k))$
with $p=3$, $c=17=(122)_3$, $h=5$, $n=17
{\times}  3^5=(12200000)_3=4131, k=241=(22221)_3$.
In this case $c_0=2 \ge k'=1$ (and $n'=243=(100000)_3$ and $M=1$ in the proof of
Theorem~\ref{th:Snkp2}). We get that $\nu_3(S(4131,241))=(\d_3(241)-1)/(3-1)=(9-1)/2=4$.

Another example determines $\nu_p(S(n,k))$ with $p=5$, $c=46=(141)_5$, $h=4$, $n= 46
{\times} 5^4=(1410000)_5=28750, k=622=(4442)_5$. In this case $c_0=1 \le k'=2$ (and $n'=3750=(110000)_5$ and $M=1$ in the proof). It follows that $\nu_5(S(28750,622))=(\d_5(622)-2)/(5-1)=(14-2)/4=3$.
\end{example}

\begin{rem}
\label{rem:stat}
Observe that the condition $n\equiv k \bmod{(p-1)}$ in part (b) of Theorem~\ref{th:Snkp2} is equivalent to $\d_p(n)=\d_p(c)
\equiv \d_p(k) \bmod{(p-1)}$. We also note that the function $\d_p(k) \bmod{(p-1)}$
is periodic in $k$ with period $p-1$, {since $k\equiv \sigma_p(k) \bmod {(p-1)}$}. If $n\equiv k \bmod {(p-1)}$ then all values
of the argument $k$ satisfying this equivalence are in the form  $k+t(p-1)$ with an
integer $t$; {thus} $k'$ remains the same  for all potential cases of $k${,} since
$k'\equiv k\equiv \d_p(k)\equiv \d_p(n)  \bmod{(p-1)}$. It follows that in the
range $0<k\le \min\{c_0,k'\}  p^\n$
with $\n$ sufficiently large, there is about a fraction \[\frac{1}{p-1}\]
of cases
covered by Theorem~\ref{th:Snkp2} according to Corollary~\ref{cor:fraction}.
\end{rem}

\begin{cor}
\label{cor:fraction}
With a fixed $c$ and the notations of Theorem~\ref{th:Snkp2}, if $n\equiv k \bmod{(p-1)}$, $0\neq c_0\equiv c
\bmod p$, $k'$
is the least positive residue of $k \bmod {(p-1)}$ (which is
the least positive residue of  $n \bmod {(p-1)}$ according to Remark~\ref{rem:stat};
thus, it is completely predetermined by $n$),
and $m=\min\{c_0,k'\}$, then
we have
\[A=\left\lfloor \frac{m p^\n-k'}{p-1}\right \rfloor+1\]
values of $k$ in the range $0<k\le m  p^\n$ satisfying
the conditions of the theorem. It implies that $\lim_{\n\to\infty} A/(m p^\n)=1/(p-1)$.
\end{cor}

\begin{rem}
\label{rem:test}
We note that it is easy to test if we have a MZC provided that $n\equiv k
\bmod{(p-1)},  k\le n,$ which is equivalent to $r=(n-k)/(p-1)\in \N$. Adelberg proved in \cite[(iii) of Theorem~2.1]{Adelberg} that $S(n,k)$ is an MZC if and only if $p\nmid {n+r\choose r}$. In this case, by definition, we can determine $\nu_p(S(n,k))$ without any direct calculation with Stirling numbers since $\nu_p(S(n,k))=(\d_p(k)-\d_p(n))/(p-1)$. For the SMZC we have the condition that $p\nmid {n-1+r \choose r}$, and then $\nu_p(S(n,k))=(\d_p(k-1)-\d_p(n-1))/(p-1)$; cf. \cite[(1.2) and Corollary 3.1]{Adseq} for the case of $p=2$ and it holds for any odd prime, too. \\

The situation is similar for the AMZC and SAMZC conditions,
although the non-trivial calculation of the maximum pole
$M$ is required; cf. Theorem~\ref{th:3.x}.
With $r=(n-k)/(p-1)\in\N$, if we add the condition that $\nu_p(k)\le \nu_p(n)$ then
we have an AMZC if and only if $\nu_p\left( {n+r\choose r} \right)=\d_p(n-k)/(p-1)-M$ where
$M$ is the maximum pole of $B_{n-k}^{(-k)} (x)$. Then, once $M$ is determined,
but without any calculations with $S(n,k)$, we obtain that $\nu_p(S(n,k))=
\nu_p\left({n\choose k}\right)-M=(\d_p(k)-\d_p(n))/(p-1)+\nu_p({n+r\choose r})$. In a similar fashion,
with $r=(n-k)/(p-1)\in\N$, if we add the condition that $\nu_p(n)< \nu_p(k)$
then we have a SAMZC if and only if $\nu_p\left( {n-1+r\choose r} \right)=\d_p(n-k)/(p-1)-M'$
where $M'$ is the maximum pole of $B_{n-k}^{(-k+1)} (x)$. Once $M'$ is determined,
but without any calculations with $S(n,k)$, we obtain that $\nu_p(S(n,k))=\nu_p
({n-1\choose k-1}
)-M'=(\d_p(k-1)-\d_p(n-1))/(p-1)+\nu_p({n-1+r\choose r})$.
\end{rem}

\begin{example}
\label{ex:moreexample}
To see the relevance of the condition $n\equiv k \bmod{(p-1)}$, we
tested some examples in which we can determine whether $S(n,k)$ is an
AMZC or SAMZC,  although it requires the exact calculation of $\nu_p(S(n,k))$
and the corresponding maximum pole for establishing these facts.

For the cases with  ${\nu_p(n)<\nu_p(k)}$ and ${p-1\nmid n-k}$ we found
that $\nu_3(S(100,45))=2$, $\nu_5(S(301,75))=2$ and $\nu_7(S(272,98))=1$
and they are SAMZ cases.
We observed in our examples that $\nu_p(n)=0$.\\

For ${\nu_p(k)<\nu_p(n)}$, for the cases we considered, we observed that the condition $n \equiv k \bmod{(p-1)}$ was always true if $S(n,k)$ is an AMZC. It is an interesting open question whether this is always true.
\end{example}

\begin{rem}
\label{rem:3.11}
Theorem~\ref{th:Snkp2} can be further generalized to give the following theorem.  This theorem is more general, since it dispenses with the condition that $k\le c_0 p^h$, and in part (b) offers an alternative test to show that $S(n,k)$ is an AMZC without apriori calculation for the maximum pole $M$. This theorem is stronger than Theorem~\ref{th:Snkp2}, and the hypotheses are more germane to the conclusion. It retains only the hypotheses that $\nu_p(k)\le \nu_p(n)$ and $p-1\mid n-k$.
Theorem~\ref{th:Snkp2} is the special case of this theorem where $B=n'$.
\end{rem}

We need the notion of segments; cf. Definition~\ref{def:bs}.

\begin{thm}
\label{th:Arnieplus}
  Let $B$ be a bottom segment of $n$  such that $0<k\le B$. Suppose that $p-1\mid B-k$ and $p\nmid {n+(B-k)/(p-1)\choose n}$.
  Then  if $M$ is the maximum pole of $B_{n-k}^{(-k)} (x)$, we have
\begin{itemize}
\item[(a)] $M=\d_p(B-k)/(p-1)$, and the first occurrence of this pole in the higher order Bernoulli polynomial  is in degree $n-k -(B-k)=$ degree $n-B$;

\item[(b)] furthermore, if
$p-1\mid n-k$ then $S(n,k)$ is an AMZC and $\nu_p(S(n,k))=
(\d_p(k)-\d_p(B))/(p-1)$.
\end{itemize}
\end{thm}

\begin{proof}[Proof of Theorem~\ref{th:Arnieplus}]
The proof is essentially identical to that of Theorem~\ref{th:Snkp2}, with the segment $n'$ replaced by $B$.  The additional information about the first occurrence of the pole of order $M$ being in degree $n-B$ follows from the general theory of maximum poles, namely if $N_1, \dots, N_M$ is the (longest) Kimura chain, then the first occurrence of the maximal pole is in degree $n-k-\sum N_i =n-k-(B-k)=n-B$.\\

For part (b), the hypothesis guarantees that $S(B,k)$ is a MZC, which implies that $\nu_p(k)\le \nu_p(B)\le \nu_p(n)$ by Theorem ~\ref{th:newth}. The proof proceeds as before, using the AMZC criterion in Theorem~\ref{th:3.x}.
\end{proof}

\begin{example}
\label{ex:plus}
We illustrate the use of Theorem~\ref{th:Arnieplus} with $p=5$, $n=2900=(43100)_5$, $k=348=(2343)_5$, $B=400=(3100)_5$, $B-k=52=(202)_5$. By part (b) we get that $\nu_5(S(2900,348))=(\d_5(348)-\d_5(400))/4=2$. Here the maximum pole is  $M=\d_5(52)/4=1$ by part (a).
\end{example}

The following theorem is essentially a restatement of Theorem~\ref{th:Arnieplus} and requires no additional proof.  It is stated as a separate theorem to indicate that it involves an invariance property of the AMZC, and is parallel to results in \cite{Adseq} for the prime $p=2$.  The assumption that $S(B,k)$ is a MZC is equivalent to $p\nmid {B+(B-k)/(p-1)\choose B}$.

\begin{thm}
\label{th:Arniex}
Let $S(B,k)$ be  a MZC.  Let $T$ be divisible by $p-1$ with all $p$-powers in $T$ at least as big as all $p$-powers in $B$, and assume that $p\nmid {T+B\choose B}$, i.e., the sum of the coefficients of $T$ and $B$ for exponent $\nu_p(T)$ is less than $p$. Then $S(T+B,k)$ is AMZC and $\nu_p(S(T+B,k))=(\d_p(k) -\d_p(B))/(p-1)$.
\end{thm}

\begin{rem}
\label{rem:Arniex}
\phantom{xxx}
\begin{itemize}
\item[(1)] If we do not assume that $p-1\mid T$, then we still get the AMZ estimate $\nu_p(S(T+B,k))\ge
(\d_p(k)-\d_p(B))/(p-1)$,
which is part (a) of all the Theorems~\ref{th:Snkp2} and \ref{th:Arnieplus}.

\item[(2)] In Theorem~\ref{th:Snkp2}, we let $n=T+B$ and $B=n'$ and $T=n-B$.  The case where $k'\le c_0$, so $B=k' p^h $ is
    Theorem~\ref{th:A2.2} for a single digit.
    The case where $k'>c_0$ uses the criterion for AMZC.
\end{itemize}
\end{rem}

\subsection{Amdeberhan-type identities}
\label{sec:Amde}
In \cite[identity (2-4)]{AMM}
 Amdeberhan et al. suggested the conjecture that $\nu_2(S(2^\n+1,k+1))=\nu_2(S(2^\n,k))$ for $1\le k\le 2^\n$, which was proven by Hong et al. in \cite[Theorem~3.2]{HZZ}. Note that in this range $S(2^\n,k)$ is a MZC.
A much shorter proof with significant extensions to arbitrary MZ cases and all primes was given by Adelberg in \cite[Theorem~2.5]{Adelberg}, using the standard recursion for Stirling numbers of the second kind.
\\

Now we turn to the general study of
\begin{\eq}
\label{eq:Amde}
\nu_2(S(n+1,k+1))=\nu_2(S(n,k)),
\end{\eq}
in order to establish conditions to guarantee the equality.
 Adelberg found that the equality \eqref{eq:Amde} is much more general than the original conjecture, namely his theorem is the following:

\begin{thm}
\label{th:newAmdeMZC}
For any prime $p$, $S(n,k)$ is  a MZC if and only if $S(n+1,k+1)$ is a SMZC.  If $S(n,k)$ is a MZC and/or $S(n+1,k+1)$ is a SMZC, {then}
$\nu_p(S(n,k))=\nu_p(S(n+1,k+1))=(\d_p(k)-\d_p(n))/(p-1)$.
\end{thm}

In fact in \cite{Adseq} the SMZC was defined for $p=2$, and in the current paper is defined for general $p$, and the criteria for these cases also shows that for any prime $p$, if $S(n,k)$
is a MZC then $S(n+1,k+1)$ is a SMZC (actually if and only if). The proofs do not involve any Stirling number recursions, but only involve the criteria. Note that the hypotheses
that $p-1\mid n-k$ and $\nu_p(k)\le \nu_p(n)$ are implied by the assumption that $S(n,k)$ is a MZC.\\

In this paper, the following theorem follows immediately from the cited theorems.
\begin{thm}
\label{th:Amde}
Let $p$ be any prime (odd or even). Assume that $\nu_p(k)\le \nu_p(n)$ and $p-1\mid n-k$. Then $S(n,k)$ is an AMZC if and only if $S(n+1,k+1)$ is a SAMZC, and if $S(n,k)$ is an AMZC then $\nu_p(S(n+1,k+1))=\nu_p(S(n,k))$.
 Similarly, if $\nu_p(k-1)\le \nu_p(n-1)$ and $p-1\mid n-k$ we get that $S(n-1,k-1)$ is an AMZC if and only if $S(n,k)$ is a SAMZC, and if $S(n-1,k-1)$ is an AMZC then $\nu_p(S(n,k))=\nu_p(S(n-1,k-1))$.
\end{thm}

This establishes the Amdeberhan-type  result, with a partial converse.
\begin{rem}
\label{rem:Amde}
In particular, to indicate how far we have generalized the original Amdeberhan conjecture, if $1\le k\le 2^\n$ and $c\ge 1$, we know that $S(c 2^\n,k)$ is an AMZC and $\nu_2(S(c 2^\n,k))=\d_2(k)-1$. Since clearly $\nu_2(k)\le \nu_2(c 2^\n)$, we have  $S(c 2^\n +1,k+1)$ is a SAMZC and $\nu_2(S(c 2^\n+1,k+1)) =\nu_2(S(c  2^\n,k))=\d_2(k)-1$. Furthermore, this is all done with no Stirling number identities or inductions!
\end{rem}

\begin{cor}
\label{cor:Amde}
Let $p$ be an odd prime. If $(n,k)$ satisfies all the hypotheses of Theorems~\ref{th:Snkp2}, \ref{th:Arnieplus}, or \ref{th:Arniex}, then $S(n+1,k+1)$ is SAMZC and $\nu_p(S(n+1,k+1))=\nu_p(S(n,k))$.
\end{cor}

To better understand the connection between the Amdeberhan conjecture and its generalizations above, we consider two formulas
$\nu_p(S(n,k))=\nu_p({n\choose k}) +\nu_p(B_{n-k}^{(-k)})$ and $\nu_p(S(n,k))=\nu_p({n-1\choose k-1})+\nu_p(B_{n-k}^{(-k+1)} (1))$ that follow from \eqref{eq:S}.\\

If we replace $(n,k)$ by $(n+1,k+1)$ in the second equation and compare the two equations we get the following theorem.\\

\begin{thm}
\label{th:3.??}
$\nu_p(S(n+1,k+1))=\nu_p(S(n,k))$ if an only if $\nu_p(B_{n-k}^{(-k)}) =\nu_p(B_{n-k}^{(-k)} (1))$.
\end{thm}


We close this section by illustrating how the last part of Theorem~\ref{th:3.??}
can be used to get the Amdeberhan-type identity in a case where $\nu_2(k)>\nu_2(n)$. We prove the following theorem, which is a strengthening of \cite[Theorem~2.4]{Adseq},
which was an improvement of \cite[Theorem~6]{DMTCS}.
It should be noted that this theorem is dual to Theorem~\ref{th:4.???1},
which involves the instances of \cite[Theorems~1.1 and 1.2]{QHs}
which we can prove by our methods.\\

\begin{thm}
\label{th:3.21}
Let $h, u, c \in N$ with $u\le 2^h$. Then if $0<u<2^h$, the SAMZ estimate is
$\nu_2(S(c 2^h+u,2^h))\ge h-1-\nu_2(u)$.  Furthermore, the SAMZ cases are $u$ even and $u\le 2^{h-1}$, and $u=1$, and $u=1+2^{h-1}$, and for these cases $S(c2^h+u-1,2^h -1)$ is an AMZC and $\nu_2(S(c 2^h+u,2^h))=\nu_2(S(c2^h+u-1,2^h-1))=h-1-\nu_2(u)$.
\end{thm}

\begin{proof}[Proof of Theorem~\ref{th:3.21}]
All but the last assertions were proven in \cite{Adseq}.  If $n=c2^h+u$ with $0<u<2^h$ and $k=2^h$, then $\nu_2(k)=h>\nu_2(n)=\nu_2(u)$.  Hence, we can use the remark following Theorem~\ref{th:3.x}
to deduce the conclusion.                                                           \end{proof}

\subsection{A statistically minded conjecture
and exceptions}
\label{sec:stat}
Inspired by
\cite[Theorem 2.6]{Adseq} and based on empirical evidence,
we claim the following somewhat surprising
\begin{conj}
\label{conj:stat}
For $c\in \Zp$ we have
\begin{linenomath}
\begin{align}
\lim_{\n\to\infty} &\frac{1}{c2^\n}\left| \left\{k: \, 0\le k\le c2^\n,
\nu_2(S(c2^\n,k))=\d_2(k)-\d_2(c)+\nu_2\left({c2^{\n+1}-k\choose c2^\n}
\right)\right\}\right|\notag \\
&=1.
\end{align}
\end{linenomath}
\end{conj}

This conjecture states that, statistically speaking, we can calculate
the exact value of the $2$-adic order of $S(c2^\n,k)$ for almost all $k$.
Note that Theorem~\ref{th:GC}
provides a coverage about 50\% of all $k$ values since $a$ is supposed to be
in the range $1\le a\le 2^{\n-1}$.

\subsection{A related result}
\label{sec:other}
In this subsection we state and proof another related result and make a conjecture for $p=2$.
\begin{thm}
\label{th:+1+2}
Let $a,b,c,\n \in \N$
with $c > b\ge 1$ and $1\le a\le 2^{\n-1}-1$.
\weset\
Then
\begin{linenomath}
\begin{align}
\label{eq:+c1}
\nu_2(S(c2^\n+1,b 2^\n+a+1))=\nu_2(S(c2^\n,b 2^\n+a))=
\d_2(a)+f(b,c),
\end{align}
\end{linenomath}
and if $3\le a\le 2^{\n-1}$ then
\begin{linenomath}
\begin{align}
\label{eq:+c2}
&\nu_2(S(c2^\n+2,b 2^\n+a))\notag \\
&
\begin{cases}
=\d_2(a)+f(b,c)+\nu_2\left({a+1\choose 2}\right)-1,
&\text{ if } \ a\equiv 0,1,2 \bmod 4, \\
\ge \d_2(a)+f(b,c), &\text{ if } \ a\equiv 3 \bmod 4,\\
\end{cases}
\end{align}
\end{linenomath}
if $2^{\n-2}\ge \nu_2\left({2c-b-1 \choose c}\right)$.
If $b\ge 1$ is odd then
\eqref{eq:+c1} simplifies to $\d_2(a)+\nu_2(c-b)$ while the first case of
\eqref{eq:+c2} simplifies to $\d_2(a)+\nu_2(c-b)+\nu_2\left({a+1\choose 2}\right)-1$.\\
\end{thm}

Numerical experimentation suggests the following conjecture. 
\begin{conj}
\label{con:a+2}
If $3\le a\le 2^{h-1}$ and $a\equiv 3 \bmod 4$ then $\nu_2(S(c2^h+2,b2^h+a))=\d_2(a)+f(b,c)+\nu_2\left({a+1\choose 2}\right)-1$.\\
\end{conj}

\begin{proof}[Proof of Theorem~\ref{th:+1+2}]
First we prove \eqref{eq:+c1}. By the standard recurrence relation for the Stirling numbers we have
\begin{\eq}
S(c 2^h+1, b 2^h+a+1)=S(c 2^h,b 2^h+a)
+(b 2^h+a+1) S(c 2^h,b 2^h+a+1).\notag
\end{\eq}
We have $\nu_2(S(c 2^h,b 2^h+a))=\d_2(a)+f(b,c)=\d_2(a+1)-1+\nu_2(a+1)+f(b,c)$
if $1\le a\le 2^{h-1}$ and $\nu_2((b 2^h+a+1) S(c 2^h,b 2^h+a+1))=\nu_2(a+1)+\d_2(a+1)+f(b,c)$ if $1\le a+1\le 2^{h-1}$, i.e., $1\le a \le 2^{h-1}-1$ combined, by Theorem~\ref{th:GC}. Therefore, $\nu_2(S(c 2^h+1,b 2^h+a+1))=\nu_2(S(c 2^h,b 2^h+a))=\d_2(a)+f(b,c)$ which also generalizes \cite[Theorem~3.2]{HZZ} that proved a conjecture by Amdeberhan et al.; cf. \cite[identity (2-4)]{AMM}.\\

To prove \eqref{eq:+c2}, we use the standard recurrence relation again and have 
\begin{\eq}
\label{eq:st}
S(c 2^h+2, b 2^h+a)=S(c 2^h+1,b 2^h+a-1)
+(b 2^h+a) S(c 2^h+1,b 2^h+a).
\end{\eq}
We apply $\nu_2(b 2^h+a)=\nu_2(a)$ and \eqref{eq:+c1}
for $\nu_2(S(c2^h+1, b2^h+a-1))$ with $1\le a-2\le 2^{h-1}-1$ and $\nu_2(S(c2^h,b2^h+a))$
with $1\le a-1\le 2^{h-1}-1$, i.e., with $3 \le a \le 2^{h-1}$ combined. We have three cases.\\

For $a$ even we get
that $\nu_2(S(c 2^h+1, b 2^h+a-1))=\d_2(a-2)+f(b,c)=\d_2(a-1)-1+\nu_2(a-1)+f(b,c)=
\d_2(a-1)+f(b,c)-1<\nu_2(S(c 2^h+1, b 2^h+a)$ for the two terms in \eqref{eq:st}.
In this case, by \eqref{eq:+c1}, $\nu_2(S(c 2^h+1, b 2^h+a-1))=
\nu_2(S(c 2^h, b 2^h+a-2))=
\d_2(a-2)+f(b,c)=
\d_2(a-1)-1+\nu_2(a-1)+f(b,c)=
\d_2(a)-1+\nu_2(a)-1+\nu_2(a-1)+f(b,c)=\d_2(a)+\nu_2({a+1\choose 2})-1+f(b,c)$.\\

 In a similar fashion, if $a\equiv 1 \bmod 4$ then $\nu_2(S(c 2^h+1, b 2^h+a-1))=
 \d_2(a-2)+f(b,c)=\d_2(a-1)-1+\nu_2(a-1)+f(b,c)\ge \d_2(a-1)+f(b,c)+1>
 \nu_2(S(c 2^h+1, b 2^h+a))$ for the two terms in \eqref{eq:st}. In this case, by
 \eqref{eq:+c1}, $\nu_2(S(c 2^h+1, b 2^h+a))=
 \nu_2(S(c 2^h, b 2^h+a-1))=
 \d_2(a-1)+f(b,c)=
 \d_2(a)-1+\nu_2(a)+f(b,c)=
 \d_2(a)+\nu_2({a+1\choose 2})-1+f(b,c)$.\\

 If $a\equiv 3 \bmod 4$, then the
  two terms in \eqref{eq:st} have the same 2-adic orders,
 so the sum has greater value, i.e., $\nu_2(S(c 2^h +2, b 2^h+a))\ge \nu_2(S(c2^h+1, b2^h+a))+1
 =\d_2(a)+\nu_2(a)+f(b,c)=\nu_2(S(c 2^h, b 2^h+a))$.
 \end{proof}


\section{Main results for $s(n,k)$}
\label{sec:main2}
\subsection{Preliminaries-Estimates and cases}
\label{sec:main2.1}
The in-depth study of $p$-adic evaluation of Stirling numbers of the first kind started by Lengyel in \cite{Lens}
who noticed the inherent and significant differences in the $p$-adic behavior of the Stirling numbers of the first and second kinds.\\

 Further progress has been made by Leonetti and Sanna  \cite{LS}, Komatsu and Young \cite{KY},  Adelberg \cite{Adelberg} and \cite{Adseq}, and Qui and Hong \cite{QHsp} and \cite{QHs}.\\

As an example of our approach, we illustrate it by consideration of the following theorems.

\begin{thm}[{\cite[Theorem~1.2]{QHs}}]
\label{th:QH2}
For arbitrary positive integers $h$ and $k$ such that $k \le 2^h$, we have
$\nu_2(s(2^h+1,k+1)) = \nu_2(s(2^h, k))$.
\end{thm}
Note that they also proved 
\begin{thm}[{\cite[Corollary~1.3]{QHs}}]
\label{th:QH3}
For arbitrary integers $h$ and $k$ such that $h\ge 2$ and $2\le k \le
2^{h-1}+1$, we have
\begin{linenomath}
\begin{align*}
\nu_2(s(2^h,2^h-k)) =
\begin{cases}
h-1-\nu_2(k), &\text{ if } 2 \mid k,\\
2h-2-\nu_2(k-1), &\text{ if } 2\nmid k.
\end{cases}
\end{align*}
\end{linenomath}
\end{thm}

Although with Adelberg's technique we can prove only limited versions of Theorems~\ref{th:QH2} and \ref{th:QH3},
our proofs are very short and transparent, and can provide the inductive basis of more general proofs.
They illustrate how the estimates and cases can be used, and also demonstrate the limitations of our methods in dealing with Stirling numbers of the first kind. The proofs in \cite{QHs} are very long and involved, with many inductions.\\

Our method for dealing with Stirling numbers of the first kind is entirely analogous to the method for dealing with Stirling numbers of the second kind, since both depend entirely on the translation from higher order Bernoulli numbers and polynomials to Stirling numbers.\\

The basic translation formulas  for Stirling numbers of the first kind are 
\begin{\eq}
\label{eq:s4.1}
s(n,k)={n-1\choose k-1} B_{n-k}^{(n)}
\end{\eq}
and the ``shifted formula" 
\begin{\eq}
\label{eq:s4.2}
s(n,k)={n\choose k} B_{n-k}^{(n+1)} (1).
\end{\eq}
Comparison of formulas \eqref{eq:s4.1} and \eqref{eq:s4.2} for Stirling numbers of the first kind with formulas \eqref{eq:new} and \eqref{eq:newshift} for Stirling numbers of the second kind, illustrates the duality between the different types of Stirling numbers,
which arises
from the duality between Stirling polynomials of the first and second kinds.
This duality basically reverses the roles of $n$ and $k$, reverses inequalities, and interchanges sums and differences. Essentially one interchanges the $n$ and $k$, and replaces them by their negatives.  This will be further illustrated by the estimates and cases, and by the applications.\\

Since most of our analysis of Stirling numbers of the first kind is almost the same as for Stirling numbers of the second kind, we will state the results, with proofs only when there are significant differences.\\


For a general prime $p$, Adelberg proved a series of interesting results in \cite{Adelberg}.
He proved the MZ estimate $\nu_p(s(n,k))\ge (\d_p(k-1) -\d_p(n-1))/(p-1)$. He showed that $s(n,k)$ is a MZC if and only if the estimate is sharp, which is in turn equivalent to $p\nmid {k-1\choose r}$, where $r=(n-k)/(p-1)\in \N$
in \cite[Theorem~3.1]{Adelberg}.  He also proved  in \cite[Theorem~3.4]{Adelberg} that if $s(n,k)$ a MZC  then the Amdeberhan-type result
$\nu_p(s(n-1,k-1))=\nu_p(s(n,k))$ holds.\\

For $p=2$ Adelberg introduced the MZ, SMZ, AMZ, and SAMZ estimates \eqref{eq:982}--\eqref{eq:sh}  for $s(n,k)$ in \cite[(3.7)--(3.10)]{Adseq}, in respective order: 
\begin{\eq}
\label{eq:982}
\nu_2(s(n,k))\ge \d_2(k-1)-\d_2(n-1),
\end{\eq}
\begin{\eq}
\label{eq:983}
\nu_2(s(n,k))\ge \d_2(k)-\d_2(n),
\end{\eq}
\begin{\eq}
\label{eq:unsh}
\nu_2(s(n,k))\ge \d_2(k-1)-\d_2(n-1)+\#([n-k]-[k-1]),
\end{\eq}
and
\begin{\eq}
\label{eq:sh}
\nu_2(s(n,k))\ge \d_2(k)-\d_2(n)+\#([n-k]-[k]).
\end{\eq}
\\
Since for the higher order Bernoulli polynomial $B_{n-k}^{(n)} (x)$, using the notations of Section 5.3, we have $s=n-(n-k)-1=k-1$, by Remark~\ref{rem:Ss}
the maximum pole $M=\#([k-1] \int [n-k])$, and similarly, the maximum pole of  $B_{n-k}^{(n+1)} (x)$ is $M'=\#([k] \int [n-k])$, so we can reformulate the AMZ (almost minimum zero) estimate and SAMZ (shifted almost minimum zero) estimate as  $\nu_2(s(n,k))\ge \nu_2({n-1\choose k-1})-M$ and $\nu_2(s(n,k))\ge \nu_2({n\choose k})-M'$, respectively.\\

If any of these estimates is sharp, we have a ``case," so again we have the cases MZC, SMZC, AMZC, and SAMZC, with the same geometric interpretations.\\

The reformulation can be taken to define the four estimates and cases for an arbitrary prime $p$. Alternatively, one can use Definition~\ref{def:new}
with $n:=n-k$ and $l:=n$.  Since in \cite{Adseq} only $p=2$ was considered, and the focus was on Stirling numbers of the second kind, the estimates and cases were given for Stirling numbers of the first kind only for $p=2$, but there was no further development.
Accordingly, we will now generalize the estimates and cases (when the estimates are sharp) for all primes $p$ for Stirling numbers of the first kind.\\

For odd $p$, if $M$ is the maximum pole of $B_{n-k}^{(n)} (x)$ and $M'$ is the maximum pole of $B_{n-k}^{(n+1)} (x)$ respectively, the MZ, SMZ, AMZ, and SAMZ estimates in respective order are given in \eqref{eq:31.5.1}--\eqref{eq:31.5.4}:
                          \begin{\eq}
                          \label{eq:31.5.1}
                          \nu_p(s(n,k))\ge (\d_p(k-1)-\d_p(n-1))/(p-1)                                              \end{\eq}
                          \begin{\eq}
                          \label{eq:31.5.2}
                          \nu_p(s(n,k))\ge (\d_p(k)-\d_p(n))/(p-1)                                                     \end{\eq}
                          \begin{\eq}
                          \label{eq:31.5.3}
                          \nu_p(s(n,k) \ge (\d_p(k-1)-\d_p(n-1))/(p-1)+(\d_p(n-k))/(p-1) -M
                           \end{\eq}
                          \begin{\eq}
                          \label{eq:31.5.4}
                          \nu_p(s(n,k))\ge (\d_p(k)-\d_p(n))/(p-1)+(\d_p(n-k))/(p-1) -M'.      \end{\eq}
                          
Observe that we have alternate formulations to the estimates \eqref{eq:31.5.3} and \eqref{eq:31.5.4}, namely $\nu_p(s(n,k))\ge \nu_p({n-1\choose k-1})-M$ and $\nu_p(s(n,k))\ge \nu_p({n\choose k})-M'$, respectively. Also $s(n,k)$ is an AMZC if and only if $\nu_p(B_{n-k}^{(n)})=-M$, and $s(n,k)$ is a SAMZC if and only if $\nu_p(B_{n-k}^{(n+1)} (1))=-M'$.\\

Note that the AMZ estimate improves the MZ estimate if that estimate is not sharp, and similarly the SAMZ estimate improves the SMZ estimate if that one is not sharp.  Also note that the shifted MZ estimate for $s(n,k)$ is the same as the MZ estimate for $s(n+1,k+1)$, and similarly, the shifted AMZ estimate for $s(n,k)$ is the same as the AMZ estimate for $s(n+1,k+1)$, so $s(n,k)$ is a SMZC if and only if $s(n+1,k+1)$ is a MZC, and if these cases hold, then $\nu_p(s(n+1,k+1))=\nu_p(s(n,k))$, which is equivalent to $\nu_p(B_{n-k}^{(n+1)})=\nu_p(B_{n-k}^{(n+1)} (1))$. This result considerably strengthens \cite[Theorem~3.4]{Adelberg}, namely if we replace $(n,k)$ by $(n-1,k-1)$, then we now have that  if $s(n,k)$ is a MZC then $\nu_2(s(n,k))=\nu_2(s(n-1,k-1))$ and $s(n-1,k-1)$ is a SMZC, which we did not have previously. In particular, we get the following improvement of \cite[Theorem~3.2]{Adelberg}:

\begin{thm}
\label{th:4.?}
Let $1\le a\le p-1$ and $k=a p^h$ and assume that $p-1\mid n-a$ and $k\le n<k p$. Then $s(n,k)$ is a MZC and $s(n-1,k-1)$ is a SMZC and $\nu_p(s(n,k))=\nu_p(s(n-1,k-1))=((\d_p(k-1)-\d_p(n-1))/(p-1)$, which is an Amdeberhan-type result.\\
\end{thm}

Observe that according to \eqref{eq:31.5.1} and \eqref{eq:31.5.2}, the theorem is the dual of De Wannemacker's theorem for Stirling numbers of the second kind as generalized to all primes in \cite[Theorem~2.2]{Adelberg}.\\

 We can also easily show as before that if $s(n,k)$ is a MZC then $\nu_p(n)\le \nu_p(k)$, while if $s(n,k)$ is a SMZC then $\nu_p(k)\le \nu_p(n)$. Finally,  we will prove the following theorem here, because the proof differs significantly from that of
 Theorem~\ref{th:newth}, which is the corresponding result for Stirling numbers of the second kind.

\begin{thm}
\label{th:4.?1}
The AMZ and SAMZ estimates for
$s(n,k)$
 are non-negative.
\end{thm}

\begin{proof}[Proof of Theorem~\ref{th:4.?1}]
 We give the proof for the SAMZ estimate.  The AMZ estimate is similar.
If $M'$ is the maximum pole of $B_{n-k}^{(n+1)} (x)$ and $N_1, \dots, N_{M'}$ is the  (longest) Kimura chain, then $p\nmid {k\choose N_i/(p-1)}$ for each $i$.  If $e=\nu_p(N_i)$ and $c$ is the coefficient of $p^e$ in the base $p$ expansion of $N_i$, then $p-c$ is the coefficient of $p^e$ in $N_i/(p-1)$, and $p\nmid {k \choose N_i/(p-1}$, so $p-c\le k_e$, for all $i$, where $k_e$ is the  coefficient of $p^e$ in $k$, i.e., $c+k_e\ge p$.  Hence $N_i+k$ has a base $p$ carry in place $p^e$ for each $i$.  This $(n-k)+k$ has at least $M'$ base $p$ carries, so $\nu_p({n \choose n-k})\ge M'$, i.e.,  $\nu_p({n\choose k})-M'\ge 0$.
\end{proof}

\begin{cor}
\label{cor:4.?1}
If $\nu_p(s(n,k))=0$  then $s(n,k)$ is a MZC or an AMZC and also
a SMZC or a SAMZC.
\end{cor}

We turn now to the criteria for the cases, based on the partitions $\uu$, defined in Section~\ref{sec:criteria}

\begin{thm}
\label{th:4.?2}
(criteria for  the four cases for $p=2$)
\
\begin{itemize}

\item [(i)] $s(n,k)$ is a MZC if and only if $2\nmid {k-1\choose n-k}$
\item [(ii)] $s(n,k)$ is a SMZC if and only if $2\nmid {k\choose n-k}$, i.e., if and only if $s(n+1,k+1)$ is a MZC
\item [(iii)] $s(n,k)$ is an AMZC if and only if precisely one of the following  holds:
\begin{itemize}
\item [(a)] $\nu_2({k-1\choose n-k}))=\d_2(n-k)-M$, where $M$ is the maximum pole of $B_{n-k}^{(n)} (x)$
\item [(b)] $\nu_2({k-1 \choose n-k-1})=\d_2(n-k)-M-1$
\item [(c)] $n-k$ is odd and $\nu_2({k-1\choose n-k-2})=\d_2(n-k)-M -1$.
    \end{itemize}
\item[(iv)] $s(n,k)$ is a SMZC if and only if precisely one of the following holds:
\begin{itemize}
\item [(a)] $\nu_2({k\choose n-k})=\d_2(n-k)-M'$, where $M'$ is the maximum pole of $B_{n-k}^{(n+1)}(x)$
\item [(b)] $n-k$ is odd and $\nu_2({k\choose n-k-2})=\d_2(n-k)-M'-1$.
\end{itemize}
\end{itemize}
Furthermore, if $\nu_2(n+1)\le \nu_2(k+1)$ then $s(n,k)$ is a SAMZC
 if and only if $s(n+1,k+1)$ is a AMZC, and if these cases hold then $\nu_2(s(n,k))=\nu_2(s(n+1,k+1))$.
\end{thm}

Up to the last statement the proof precisely follows the pattern of
\cite[Theorem~3.3]{Adseq} where the criteria are proved for Stirling numbers of the second kind. The last statement follows from Theorem~\ref{th:5.x1-5.11} (c),
by setting its parameters $l$ and $n$ to $l:=n$ and $n:=n-k$.  The basis of the proof is that only the three canonical partitions $u_1 =n-k$, $u'_1=n-k-1$,  and $u''_1=n-k-3$, $u''_3 =1$ have to be considered, and the partition $\uu'$ is ruled out for $B_{n-k}^{(n)}(1)$ by the shift,  and is also ruled out for  $B_{n-k}^{(n+1)}$ by the assumptions about $\nu_2(n+1)$,  as in Theorem~\ref{th:5.x1-5.11} (c).\\

\begin{rem}
\label{rem:4.?2}
We can replace the valuation conditions for the binomial coefficients   in (a), (b), (c) by assumptions on $\nu_2(n-k)$ and $\nu_2(n-k-1)$ and the assumption  that the binomial coefficients have no unforced borrows, which is analogous to the assumption of no unforced carries that we made in \cite{Adseq}. (An unforced borrow is one that is a consequence of a previous borrow.)
\end{rem}

The situation for $p$ odd is almost identical to $p=2$, except there is no analog of the partition $\uu''$.  Once again, we can prove it in similar fashion to previous proofs, or deduce it from Theorem~\ref{th:5.x}.
Assume $r=(n-k)/(p-1) \in \N$, and $M$ and $M'$ are the maximum poles of $B_{n-k}^{(n)} (x)$ and $B_{n-k}^{(n+1)} (x)$ respectively, and  $\uu$ is the partition with $u_{p-1}=r$.

\begin{thm}
\label{th:4.???}
Let $p$ be an odd prime. Then
\begin{itemize}

\item [(a)] $s(n,k)$ is a MZC if and only if $p\nmid {k-1\choose r}$.
\item [(b)] $s(n,k)$ is a SMZC if and only if $p\nmid {k\choose r}$.
\item [(c)] If $l:=n$ and $n:=n-k$, then the assumption $\nu_p(l)\le \nu_p(n)$ of Theorem~\ref{th:5.x}
    says that $\nu_p(n)\le \nu_p(k)$, so if this holds then $s(n-1,k-1)$ is a SAMZC if and only if $s(n,k)$ is an AMZC, and if these cases hold then $\nu_p(s(n,k))=\nu_p(s(n-1,k-1))$.
\item [(d)] With the same notations as in (c), if instead $\nu_p(l)>\nu(n)$  the $\nu_p(k)<\nu_p(n)$, then we can replace {$(n,k)$ by $(n+1,k+1)$ and we get $s(n,k)$} is a SAMZC if and only if $s(n+1,k+1)$ is an AMZC, and if these cases hold, then $\nu_p(s(n,k))=\nu_p(s(n+1,k+1))$.
\end{itemize}
\end{thm}

Note the Amdeberhan-type identities for Stirling numbers of the first kind.\\

\subsection{Applications}
\label{sec:main2.2}
As an example of the power and limitation of our methods, we turn to the remarkable paper by Qui and Hong \cite{QHs} which derives formulas for $\nu_2(s(n,k))$ if $n=2^h$ and $0<k\le n$. This paper is quite long and involved and quite remarkable for its skill and imagination.  Our methods are limited since for the Stirling numbers of the first kind $s(n,k)$, we know that $s=k-1$ for the MZ and AMZ estimates, while $s=k$ for the SMZ and SAMZ estimates.
Thus for $p=2$,  the only partitions that must be considered for the AMZC are the three canonical partitions we have previously considered for Stirling numbers of the second kind,  while for the SAMZC, we have just the two canonical partitions with weight $n-k$.  Since $s=n-(n-k)-1=k-1$ for the unshifted estimate and $s=n-(n-k)=k$ for the shifted estimate, and $d=n-k$, $n-k-1$, or $n-k-2$ for the canonical partitions, we see that an AMZC can only occur is $n-k-2\le k-1$, i.e., $n\le 2k+1$.  Similarly, a SAMZC can only occur if $n\le 2k+2$.  Note that  if the canonical partition with $d=n-k-2$ gives the maximum pole, then $n-k$ is odd.  We have an easy proof of the following theorem, which illustrates our method.

\begin{thm}
\label{th:4.???1}
 For $p=2$, if $n=2^h$, then $s(n,k)$ is a SAMZC if and only if $k$ is even and $2^{h-1}\le k< 2^h$, or $k=2^h-1$, or $k=2^{h-1} -1$. There are no MZ or  AMZ cases if $0<k<n$. If $k$ is even and $2^{h-1}\le k\le 2^h$, then $\nu_2(s(n,k))=h-\nu_2(k)-1$. If $k=2^h-1$ or $k=2^{h-1}-1$ then $\nu_2(s(n,k))=h-1=h-1-\nu_2(k)$.  If $s(n,k)$ is a SAMZC
 then  we have the Amdeberhan equation $\nu_2(s(n+1,k+1))=\nu_2(s(n,k))$.
\end{thm}

\begin{proof}[Proof of Theorem~\ref{th:4.???1}] A simple computation shows that $M$, the maximum pole of $B_{n-k}^{(n)} (x)$ is $\#([k-1] \int [n-k])=0$ and $M'$ is the maximum pole of $B_{n-k}^{(n+1)} (x)=\#([k] \int [n-k])=1$, since $[k-1]$ is disjoint from $[n-k]$ and $[n-k] \int [k] =\{2^{\nu_2(k)}\}$.  Since $\nu_2({2^h -1\choose k-1})=0$, the AMZ estimate of $s(n,k)$ is always the trivial estimate $\nu_2(s(n,k))\ge 0$. On the other hand, it is easy to see that $\nu_2({n\choose k})=h-\nu_2(k)$, and since $M'=1$, the  SAMZ estimate is $\nu_2(s(n,k))\ge h-\nu_2(k)-M'=h-\nu_2(k)-1$.\\

We noted above that if $s(n,k)$ is any of the cases, then $k\ge n/2 -1=2^{h-1}-1$. We first dispose of the odd cases. If $k=2^h -1=n-1$, then we know that $\nu_2(s(n,k))=\nu_2(n (n-1)/2)=h-1$. Since $2\nmid {k\choose n-k}$, this is a SMZC. Next if $k=2^{h-1}-1$, then $n-k=2^{h-1}+1>k$. Thus, only the partition $\uu''$ is viable, with $n-k-2=2^{h-1}-1=k$, and the criterion for SAMZC is satisfied for this partition,  so $s(n,2^{h-1} -1)$ is a SAMZC and $\nu_2(s(n,k))=\nu_2({n\choose k})-M'=\nu_2({2^h\choose 2^{h-1} -1})-1 =h-1$.\\

Next we consider the even  cases $2^{h-1}\le k\le 2^h$.  Since $n+1$ is odd, we have $\nu_2(n+1)\le \nu_2(k+1)$, so we have only to consider the partitions $\uu$ and $\uu''$. Since $k$ is even, so is $n-k$, and $\nu_2(k)=\nu_2(n-k)$, but for all other places, the base 2 ones of $n$ and $n-k$ are disjoint.  It follows ${k\choose n-k}$ has no unforced borrows, i.e., $\uu$ satisfies  the SAMZC criteria. Since $k$ is even, so is $n-k$, so the partition $\uu''$ does not give the maximum pole, which implies that $s(n,k)$ is  a SAMZC. Since we showed above that the SAMZ estimate is $\nu_2(s(n,k))\ge h-\nu_2(k)-1$, we have $\nu_2(s(n,k))=h-\nu_2(k)-1$ for these cases.\\

If $2^{h-1}<k<2^h-1$ and $k$ is odd then $k$ and $n-k$ differ in all places except the unit place and $n-k$ is now odd. It follows that the partitions $\uu$ and $\uu''$  both satisfy the conditions of the criteria, which shows that $s(n,k)$ is not a SAMZC (violates unique partition satisfying the criteria).\\

Finally, if $s(n,k)$ is a SAMZC, since $n+1$ is odd, we have $\nu_2(n+1)\le \nu_2(k+1)$, we have $s(n+1,k+1)$
is an AMZC and $\nu_2(s(n,k))=\nu_2(s(n+1,k+1))$ by Theorem~\ref{th:4.???} (d).
\end{proof}

Theorem~\ref{th:4.???1}
is another good example of the duality between Stirling numbers of the first and second kinds.
Theorem~\ref{th:3.21}
  asserts that  if $c, u, h>0$ and $u<2^h$ then $S(c 2^h+u,2^h)$
is a SAMZC if and only if $u$ even and $u\le 2^{h-1}$, or $u=1$, or $u=1+2^{h-1}$, and in these cases $\nu_2(S(c2^h+u,2^h))=h-1-\nu_2(u)$.
%
 First observe that $c=1$ is sufficient for this result by an invariance property of SAMZ cases (cf. \cite[Theorem~3.5]{Adseq}). Then observe that if you interchange $k$ and $n$, and replace $k$ by $n-u$, you get exactly the same SAMZ cases for $S(2^h+u,2^h)$ and $s(2^h,2^h-u)$, with the same $2$-adic values.\\


{\section{Appendix--basic facts, notations, techniques, and supplementary materials}
\label{sec:appendix}
In this section we collected some basic definitions (cf. \cite{Acta}, \cite{JNT}, \cite{Cohen2} and \cite{Gouvea})
and techniques. We assume that $a, a_i, b, B$, $k, l, n, n_i$, $m, M, r, r_i, s, T\in \N$  unless otherwise
indicated and that $p$ is a prime. We also define the cases and estimates for higher order Bernoulli polynomials, and prove some fundamental theorems.

\subsection{Base $p$ arithmetic, $p$-adic valuations, binomial coefficients}
\label{sec:pary}
Let $n= \sum_{i=0}^m n_i p^i$, $0\le n_i\le p-1$, be the base $p$ representation of $n$, where the coefficients $n_i$ are the base $p$ digits of $n$. We use the
notation $\d_p(n)=\sum_{i=0}^m n_i$
for the digit sum  in the base $p$ representation. The following elementary facts about base $p$ arithmetic are used often in this paper:

\begin{\eq}
\label{eq:A-1}
\d_p(a-1)=\d_p(a)-1+(p-1)\nu_p(a)
\end{\eq}
if $a\in \Zp$. \\

By theorems by Legendre and Kummer
we have $\nu_p({a+b\choose a})=(\d_p(a)-\d_p(a+b ) +\d_p(b))/(p-1)$, which is the number of carries for the base $p$ sum of $a+b$.
In particular, $\nu_p({a+b \choose a})=0$ happens if and only if $a_i +b_i <p$ for all $i$.\\

\noindent Let $r=n/(p-1) = \sum r_i p^i\in \N$ with $0\le r_i\le p-1$.  Then
\begin{itemize}
\item[(1)] $\nu_p(n)=\nu_p(r)$;

\item[(2)] $r+n=p r$, so if $i=\nu_p(n)$ then $r_i =p-n_i$;

\item[(3)] $\nu_p({n+r\choose r})=\d_p(n)/(p-1)$;

\item[(4)] if $T$ is a top segment (cf. Definition~\ref{def:bs}) of $n$ such that $p-1\mid T$, and $r=T/(p-1)$, then $\nu_p({n+r\choose r})=\nu_p({T+r\choose r})=
\d_p(T)/(p-1)$.
\end{itemize}

Here in items (2) and (3) we used that $r+n=(n+(np-n))/(p-1)=n p/(p-1)=p r$ and therefore, $\nu_p({n+r\choose r})=(\d_p(n)-\d_p(n+r)+\d_p(r))/(p-1)=\d_p(n)/(p-1)$.\\

Incidentally, the use of (2) and (3) is critical for our method, namely almost everything in the approach is reduced to the $p$-adic valuation of binomial coefficients: we express a Stirling number (of first or second kind) as a product of a binomial coefficient and a higher order  Bernoulli number.  The $p$-adic value of the Stirling number is the sum of the $p$-adic values of these factors.  We then express the higher order Bernoulli number as a sum of terms that depend on partitions.  The terms are primarily products of binomial coefficients and factorials, so their values can be calculated.  We can get an estimate (lower bound) if we can estimate the values of the terms, and we can get an exact value for the Stirling numbers if there is one term which has least value, which can be calculated.\\

We need the notion of segments.

\begin{defn}
\label{def:bs}
 We often identify $n$ with its base $p$ expansion  or representation,  $n=\sum
 n_i   p^i$, where the range can be taken from 0 up or from $\nu_p(n)$ up and the digits $n_i$ satisfy $0\le n_i\le p-1$. A closed segment is a subsum  for  $a\le i \le b$.  For a general segment, we add a summand $n_{a-1}^{'} p^{a-1}$ and/or a summand $n_{b+1}^\tttt p^{b+1}$ where $n_{a-1}^{'}<n_{a-1}$ and $n_{b+1}^\tttt <n_{b+1}$.
 If $a=0$ or equivalently,  if $a=\nu_p(n)$ a segment is called a bottom segment.  These bottom segments can be closed or not, even in a given context, but they do include the lowest term in the base $p$ expansion.   Thus, in the
 Kimura
 algorithm (cf. Definition~\ref{def:Kimura})
 to find the maximum pole for $B_n^{(l)} (x)$, we consider  (closed) bottom segments $S$ of the base p expansion of $n$, and examine $N=N(n-S;p)$ which is the
 smallest non-zero
 segment of $n-S$ which is divisible by $p-1$, and we take the
 smallest $S$ such that $p\nmid {s\choose N/(p-1)}$, where $s=l-n-1$. This segment
 will generally not be closed, since you
 proceed through the digits until you get $p-1$ as a digit sum.

Similarly, a top segment of $n$
is a segment that includes the summand  $n_t p^t$ where $t$ is the greatest
exponent in the base $p$ expansion of $n$.  If $B$ is a bottom segment of $n$, then $n-B=T$ is the (complementary) top segment, and $T$ is closed if and only if $B$ is closed.
\end{defn}

\subsection{Stirling numbers and higher order Bernoulli numbers and polynomials}
\label{sec:SsB}
We now list the definitions of Stirling numbers, higher order Bernoulli polynomials and numbers, and their relationships.\\

\begin{defn}
\label{def:Stirling}
For $n\in \Zp$, the Stirling numbers of the first kind $s(n,k)$ are defined by
\[(x)_n=x(x-1) \cdots (x-n + 1) =
\sum_{k=0}^n
(-1)^{n-k} s(n, k)x^k\]
while the Stirling numbers of the second kind $S(n,k)$  are defined by
\[x^n=\sum_{k=0}^n S(n, k) (x)_k.\]
\end{defn}
\begin{defn}
\label{def:bern}
The higher order Bernoulli polynomials $B_n^{(l)}(x)$ are defined by their exponential
generating function
\[\left(\frac  {t}{e^t-1}\right)^l e^{tx}=\sum_{n\ge 0} \frac{B_n^{(l)}(x)}{n!} t^n,\]
and the higher order Bernoulli numbers $B_n^{(l)}$ by $B_n^{(l)} = B_n^{(l)}(0)$.
\end{defn}
Note that $B_n^{(l)}(x)$ is a monic rational polynomial of degree $n$.\\

There is a recursion (shift) formula $B_n^{(l)} =(l/(l-n))B_n^{(l+1)}(1)$ which
together with the standard formulas for Stirling numbers (cf.
\cite{Fib98}, \cite[identity (13)]{JNT}, and \cite{Adseq}) give 
\begin{\eq}
\label{eq:S}
S(n, k) = {n\choose k} B_{n-k}^{(-k)}={n-1\choose k-1} B_{n-k}^{(-k+1)}(1).
\end{\eq}
and 
\begin{\eq}
\label{eq:s}
s(n, k) = {n-1\choose k-1} B_{n-k}^{(n)}={n\choose k} B_{n-k}^{(n+1)}(1)\\
\end{\eq}

\subsection{Poles, Newton polygons, and the cases of higher order Bernoulli polynomials}
\label{sec:5.3}
\begin{defn}
\label{def:pole}
If $\nu_p(a)=-M$, we  say that $a$ has  pole  of order $M$.  For any rational
polynomial $f(x)$ we define the maximum pole of $f(x)$ as the highest order pole of its coefficients, which is the highest power of $p$ in the denominator of any coefficient.
\end{defn}

We use Newton polygons to convey information about the poles of a rational polynomial. Here we use the following

\newcommand{\X}{x}

\begin{defn}
\label{def:Npoly}
Let $f(\X)$
be a
rational
polynomial of degree $n$,
$f(\X) =
a_0 \X^n+a_1 \X^{n-1}+\dots+a_n$
and plot the {lattice} points
$(i,\nu_p(a_i)), 0\le i \le n$.
The Newton polygon
of $f(\X)$ is the lower
boundary of the {\it convex hull} of the set of these lattice points.
\end{defn}

We are mainly concerned with the monic polynomial $B_n^{(l)}(x)$, whose Newton polygon
has $(0,0)$ as its initial vertex, and the degrees and coefficients are read from the highest down.\\

\begin{defn}
\label{def:Kimura}
 We define the Kimura function $N(n;p)$ to be the smallest segment of
 the base $p$ expansion of $n$ with digit sum $p-1$.
($N(n;p)$ is only defined if $\d_p(n)\ge p-1$.)
   For $p=2$, it is convenient to let $[n] =$ set of all 2-powers in its binary expansion of $n$, so that
   $\d_2(n)=\#([n])$, and $N(n;2)=\min\{[n]\}=2^{\nu_2(n)}$.
   \end{defn}

   The poles and their locations can be found by determining
   the (longest)
   Kimura chain as follows (cf. \cite{Acta}, \cite{Fib98}):  viewing $B_n^{(l)} (x)$ from the highest degree down,
   the first pole has order 1. If $s=l-n-1$, the first pole occurs in
   degree $n-N_1$ where $N_1 =N(n-S_1;p)$ and $S_1$ is
   the smallest
   segment of $n$ such that $p$ does not divide ${s \choose N_1/(p-1)}$.
    Similarly, the next higher pole has order 2 and occurs in
    degree $n-N_1-N_2$, where $N_2$
    is defined by minimality, having the condition that $N_2 =N(n-N_1 -S_1-S_2;p)$
    and $p$ does not divide ${s\choose N_2/(p-1)}$.  Continue the process as long as possible.
    Then the maximum pole is the
    number of $N_i$, and this pole $M$
    first occurs
    in degree $n-\sum_{i=1}^M N_i$. \\

Not only does this algorithm determine the maximum pole, but it determines the first occurrence of each pole and the vertices of the descending portion of the Newton polygon.
The poles occur in increasing order $1, 2, \dots, M$  and the first occurrence of the pole of order $j$ is in degree $n-\sum_{i=1}^j N_i$, i.e., the $j$-th vertex of the Newton polygon is $(\sum_{i=1}^j  N_i,-j)$.\\


   Since $\d_p(N_i) =p-1$, {for every $i$,} it follows that $M\le \d_p (n)/(p-1)$,
    which also shows that $B_n^{(l)}$ has a pole of order $\d_p(n)/(p-1)$ if and only if $p$
    does not divide ${s \choose n/(p-1)}$, where $s=l-n-1$. If there is such a pole, it occurs
    only in the constant coefficient $B_n^{(l)}$; cf. \cite{Acta}.\\

    Note that if $a$ is the highest degree in which the coefficient has an order $M$ pole, then $a=n-\sum_{i=1}^M N_i$, so $\d_p(a)=\d_p(n)-{M\cdot(p-1)}$,
    which implies that the AMZ estimate for $S(n,k)$ can be restated as $\nu_p(S(n,k))\ge
(\d_p(k)-\d_p(n))/(p-1)+\d_p(n-k)/(p-1) -M=(\d_p(k)-\d_p(n))/(p-1)+\d_p(a)/(p-1)$.  This confirms the result that the AMZ estimate improves the MZ estimate in all cases which are not MZC (i.e., where $a\neq 0$). The AMZ estimate for $s(n,k)$ can be similarly restated.

\begin{rem}
\label{rem:Ss}
If we consider the higher order Bernoulli polynomial associated
with $S(n,k)$, i.e., $n:=n-k$ and $l{:}=-k$, then $s=-k-(n-k)-1=-n-1$.
Then if $d_i=N_i/(p-1)$ then ${s \choose N_i /(p-1)}=(-1)^{d_i} {n+N_i/(p-1)\choose n}$.
The situation for $p=2$ is much easier than for odd primes: the poles of
$B_{n-k} ^{(-l)}(x)$ correspond, in increasing order to the set difference
$[n-k] -[n] =[n-k] -[n]\int[n-k]$, where the elements are arranged in increasing
order, and the maximum pole $M=\#([n]-[n]\int[n-k])$.\\

If we consider the Stirling number $s(n,k)$
with $n:=n-k$ and $l:=n$, then $s=k-1$, so the pole condition is
$p\nmid {k-1 \choose N_i/(p-1)}$. Thus, if $p=2$, the maximum pole $M$ is $\#([n-k] \int [k-1])$, and the almost minimum zero estimate is $\nu_2(s(n,k))\ge \d_2(k-1)-\d_2(n-1)+\#([n-k]-[k-1])$. \\
\end{rem}

Adelberg defined in \cite{Fib98} the concept of ``maximum pole case" (MPC) for higher order Bernoulli polynomials,
when the maximum pole of $B_n^{(l)} (x)$ has order $\d_p(n)/(p-1)$, which
led to the concept of ``minimum zero case" (MZC) for Stirling numbers of both kinds in \cite{Adelberg}.  We now have the concepts of  almost minimum zero, shifted minimum zero and shifted almost minimum zero cases for $p=2$ for Stirling numbers of both kinds (cf. \cite{Adseq}), and we invert the process by now defining the concepts of almost maximum pole case, shifted maximum pole case, and shifted almost maximum case for higher order Bernoulli polynomials for general primes $p$.

\begin{defn}
\label{def:new}
 $B_n^{(l)} (x)$ is an almost maximum pole case if the pole of $B_n^{(l)}$ is the maximum pole of $B_n^{(l)} (x)$;  $B_n^{(l)} (x)$ is a shifted maximum pole case if the pole of  $B_n^{(l+1)}(1)$ has order $\d_p(n)/(p-1)$,
  which is equivalent to the pole of $B_n^{(l+1)}$
has order $\d_p(n))/(p-1)$;
  $B_n^{(l)}(x)$ is a shifted almost maximum pole case if the pole of $B_n^{(l+1)}(1)$ is the maximum pole of $B_n^{(l+1)}(x+1)$, which is the maximum pole of $B_n^{(l+1)}(x)$.
\end{defn}

\begin{rem}
\label{rem:newer}
 Each of these cases corresponds to an estimate, namely  the MP and SMP cases  correspond to the estimates $\nu_p(B_n^{(l)})$ and $\nu_p(B_n^{(l+1)} (1)) \ge -\d_p(n)/(p-1)$, while the AMP and SAMP cases correspond to the estimates $\nu_p(B_n^{(l)})\ge -M$ and $\nu_p(B_n^{(l+1)} (1)) \ge -M'$, where $M$ is the maximum pole of $B_n^{(l)} (x)$ and where $M'$ is the maximum pole of $B_n^{(l+1)} (x)$.
 Furthermore, using the standard translations \eqref{eq:S} and \eqref{eq:s} from higher order Bernoulli numbers and polynomials to Stirling numbers of each kind, these cases and estimates correspond to the cases and estimates for  the Stirling numbers of each kind.
\end{rem}

\begin{rem}
\label{rem:new}
The motivation for the shifted cases is the recursion formula for higher order Bernoulli numbers.  It should be noted that since the strictly decreasing portion of the Newton polygon of $B_n^{(l+1)} (x+1)$ is the same as the strictly decreasing portion
of the Newton polygon of $B_n^{(l+1)} (x)$, these two polynomials have the same maximum pole. Thus, we have geometric interpretations of each of the four cases:
\begin{itemize}
    \item the maximum pole case (MPC) says that the Newton polygon of $B_n^{(l)} (x)$ is strictly decreasing;
    \item {} the almost maximum pole case (AMPC) says that the Newton polygon of $B_n^{(l)} (x)$ has a horizontal final segment;
    \item {} the shifted maximum pole case (SMPC) says that the Newton polygon of $B_n^{(l+1)}(x+1)$ (or of $B_n^{(l+1)}(x)$) is strictly decreasing;
    \item {} the shifted almost maximum pole case (SAMPC) says that the Newton polynomial of $B_n^{(l+1)}(x+1)$ has a horizontal final segment.
 \end{itemize}
%
%
It should be noted that our terminology of poles and maximum poles, which we think is illuminating, is not completely standard. An alternative is to extend the $p$-adic valuation to $\Q[x]$
by
$\nu_p(f(x))=\min \{\nu_p(a_i)\}$
if $f(x)=\sum_{i} a_i  x^{n-i}$.
We could then express each of the cases in terms of this extension, e.g., the AMPC says that $\nu_p(B_n^{(l)} (x))=\nu_p(B_n^{(l)})$.  This approach necessitates dealing with negative signs for the poles, and appears to us to be less constructive.\\
\end{rem}

\subsection{Partitions and criteria for the four cases}
\label{sec:criteria}
In order to determine the poles of a  higher order Bernoulli polynomial $B_n^{(l)}(x)$, where the order $l \in  \Z$, we need some preparation to determine the $p$-adic valuations of higher order Bernoulli numbers and the sum of the coefficient of higher order Bernoulli polynomials, in terms of sums based on partitions, which will be given in \eqref{eq:partition} and \eqref{eq:partition2}.
If $\uu = (u_1, u_2, \dots)$ is a sequence of natural numbers eventually zero, we
regard $\uu$ as a partition of the number $w = w(\uu) = \sum iu_i$, where $u_i$
is the multiplicity of the part $i$ in the partition and $d = d(\uu) = \sum u_i$ is the number of
summands. If $s=l-n-1$, we define $t_\uu = t_\uu(s)={s \choose d}{d\choose \uu}/\Lambda^\uu$,
where ${d \choose \uu}= {d \choose
u_1 u_2 \dots}$ is a multinomial coefficient, and
$\Lambda^\uu = 2^{u_1} 3^{u_2} \cdots$. It is also convenient to let $\nu_p(\uu) = \nu_p(\Lambda^\uu) = \sum u_i \nu_p(i + 1)$; cf. \cite[(9)]{JNT}. \\
%

In  \cite[Theorem~3.3]{Adseq} Adelberg found three critical partitions for $p=2$ that determine whether $B_n^{(l)} (x)$ is a AMZC or a SAMZC, namely $u_1=n$, or $u_1=n-1$, or $u_1=n-3$ and $u_3 =1$.  We will see that the situation is slightly simpler when $p$ is an odd prime, since there are only two critical partitions that are relevant to consideration of the cases.  For $p=2$  the maximum pole for $B_{n-k}^{(-k)} (x)$ is $\#([n-k] -[n])$, while for $B_{n-k}^{(n)}$ the maximum pole is $\#([n-k] \int [k-1])$.
%
For $p=2$ an equivalent formulation to  \cite[Theorem~3.3]{Adseq} shows that $S(n,k)$ is an AMZC if and only if exactly one of the following conditions hold, with $r=n-k$:
\begin{itemize}
\item[(1)] $\nu_2({n+r\choose n})=\#([n] \int [r])$;
\item[(2)] $\nu_2({n+r-1\choose n})=\#([n] \int[r]) -1$ (which implies that $\nu_2(n)<\nu_2(k))$;
\item[(3)] $n$ is even and $k$ is odd, and $\nu_2({n+r-2\choose n})=\#([n] \int[r]) -1$.
\end{itemize}
This reformulation shows why all three conditions cannot hold simultaneously, and in fact, (2) and (3) cannot both be true.  There are similar results for the Stirling numbers $s(n,k)$ of the first kind, as well as for the shifted estimates and cases when $p=2$; cf. \cite{Adseq}. \\

The key formulas
for us
are the explicit representations of $B_n^{(l)}$ and $B_n^{(l)} (1)$ in terms of partitions, namely 
\begin{\eq}
\label{eq:partition}
B_n^{(l)} =(-1)^n n! \sum_{w(\uu)\le n} t_\uu(s)
\end{\eq}
where $s=l-n-1$, and the corresponding formula for 
\begin{\eq}
\label{eq:partition2}
B_n^{(l)}(1)=(-1)^n n!
\sum_{w(\uu)= n} t_\uu(s);
\end{\eq}
 cf. \cite{Acta} and \cite[formulas (10) and (11)]{JNT}.\\

To translate the results obtained for
higher order Bernoulli polynomials to Stirling numbers, we use the specialization
$n:=n-k$, and $l:=-k$ or $l:=n$ in \eqref{eq:S} and \eqref{eq:s} for the Stirling numbers
of second or first kind, respectively.  Note that $s=
-k-(n-k)-1=-n-1$ for the Stirling numbers of
the second kind and ${s\choose d}=(-1)^d {n+d\choose n}$. For the Stirling numbers
of the first kind, $s=
n-(n-k)-1=k-1$ and ${s\choose d}={k-1 \choose d}$.\\

 Adelberg's analysis of the cases in \cite{Adseq} was actually based on the partitions $\uu$, with $w(u)\le n-k$ or $w(u)=n-k$ for the shifted cases. He showed that for $p=2$  there are the three critical (or canonical) partitions
 noted above
 that determine if we have an AMZC or a SAMZC for Stirling numbers, and one partition for the MZC and SAMZC.  We will now demonstrate that the situation is somewhat simpler for odd $p$, in that there are only two critical  partitions that determine the case, and one of these partitions can be often ruled out,  so that   it will suffice to test the remaining
 critical partition.\\

We consider all
partitions $\uu$ with $w(\uu)\le n$ and  set $s=l-n-1$.  We define the companion sequence $\tau_\uu(s)=\tau_\uu =(n)_w t_\uu=(n!/(n-w)!) t_\uu=w! {n\choose w} t_\uu$.
If $M$ is the maximum pole of the coefficients of $B_n^{(l)}(x)$, then $M$ is the maximum pole of $\{\tau_\uu \mid w\le n\}$, i.e.,
$\min\{\nu_p(\tau_\uu) \mid w\le n\} =-M$; cf. \cite{Acta}.  Since $n!=(n)_w (n-w)!$, we have $n! t_\uu =(n-w)!
 \tau_\uu$, so $\nu_p(n! t_\uu)\ge -M$, with equality if and only if  $\nu_p(\tau_\uu)=-M$ and $p\nmid (n-w)!$, i.e., $w\ge n-(p-1)$.
 {We also see that if $\nu_p(n! t_\uu) =-M$, in which case we say $\uu$ works, then $\uu$ minimizes both $\{\nu_p(t_\uu) \mid w\le n\}$ and $\{\nu_p(\tau_\uu) \mid w\le n\}$.
By the same kind of argument that appears in \cite{Acta} and \cite{Adseq}
if $r=n/(p-1)\in \N$, the only a priori candidate partitions $\uu$ having  $\nu_p(n! t_\uu)=-M$ are the partitions when $u_{p-1}=r$, or $u_{p-1}=r-1$ (many partitions with $w \le n$), or $u_{p-1}=r-2$ and $w=n-(p-1)=(r-1) (p-1)$ (again many partitions).
The first partition is concentrated in place $p-1$, while the other types of partitions are concentrated in places $1,...,p-1$.
We first show that the last types of partitions, when $u_{p-1} =r-2$ and $w=n-(p-1)$ do not work: obviously $\sum_{i<p-1} i u_i  =p-1$, so  $\sum_{i<p-1} u_i\ge 2$.
If $\uu$ is the partition with $u_{p-1}=r$,
and  $\uu'$ is any partition of
the last
 type,
we let $d'=d(\uu')$. Then $d'\ge r=d(\uu)$. Hence $\nu_p(t_{\uu'}) = \nu_p((s)_{d'} /(u'_1 ! \dots u'_{p-2}!(r-2)!)) - (r-2) >
\nu_p((s)_r/r!) -r =\nu_p(t_\uu)$.
(Note that $d' \ge d$ is critical and follows from above.  Also $(r-2)!=r!/(r (r-1))$, so $\nu_p((r-2)!)\le \nu_p(r!)$, and all $u_i <p$ for $i<p-1$, so $\nu_p(u_i!)=0$.)  Thus, $\nu_p(t_{\uu'})>\nu_p(t_\uu)$, so $\uu'$ cannot
give the maximum pole.

Next, consider partitions of the second type $\uu\tt$, i.e., when $u_{p-1}^{\tttt}=r-1$, which are not concentrated in place $p-1$. Then $d''=d(\uu'')\ge r$ for these partitions, and the same argument shows that $\nu_p(\tau_\uu)>\nu_p(\tau_{\uu\tt})$, so $\uu\tttt$ cannot give the maximum pole.

Hence the only partitions $\uu$ with $w\le n$ that can satisfy $\nu_p(\tau_\uu)=-M$ are the two partitions concentrated in place $p-1$, with $u_{p-1}=r$ or $u_{p-1}=r-1$.

}

Finally, we show that for any prime $p$, including $p=2$, if $\nu_p(l)\le \nu_p(n)$, and $\uu$ is the partition with $u_{p-1}=r$ and $\uu'$ is the partition concentrated in place $p-1$ with $u'_{p-1} =r-1$ then
$\nu_p(t_{\uu'})>\nu_p(t_\uu)$: since
$s=l-n-1$, we have $t_\uu/t_{\uu'}=({s\choose r}/p^r)/
({s\choose r-1}/p^{r-1} ) =(1/p) (l-n-r)/r=(1/p) (l-p n/(p-1))/(n/(p-1))=(1/p) ((p-1) l
-p n)/n$.  But $\nu_p(l)\le \nu_p(n)$ implies that $\nu_p((p-1) l-p n)=\nu_p(l)\le \nu_p(n)$, so $\nu_p(t_\uu/t_{\uu'}) <0$; thus, $t_\uu$ has a bigger pole than $t_{\uu'}$, and $\uu$ is the only candidate with $w(\uu)\le n$ for $\nu_p(n! t_\uu) =\nu_p(\tau_{\uu})
=-M$.
\\

 Thus, we have proven the following theorem, which is one of the most important results of this paper, since it leads directly to significant theorems for the $p$-adic valuations of Stirling numbers of both kinds.

\begin{thm}
\label{th:5.x}
Let $p$ be an odd prime, and assume that $r=n/(p-1) \in \N$. Let $M$ be the  maximum pole of $B_n^{(l)} (x)$ and $M'$ be the maximum pole of $B_n^{(l+1)} (x)$, and let $\uu$ be the partition with $u_{p-1}=r$.  Then
\begin{itemize}
\item[(a)] $B_n^{(l-1)} (x)$ is a SAMPC if and only $\nu_p({l-n-1\choose r})=\d_p(n)/(p-1) -M$, which is equivalent to $\nu_p(\tau_\uu) =-M=\nu_p(B_n^{(l)} (1))$.
\item[(b)] If $\nu_p(l)\le \nu_p(n)$ then $B_n^{(l)} (x)$
 is a AMPC if and only if $\nu_p(\tau_\uu)=-M=\nu_p(B_n^{(l)})$, which is equivalent to $\nu_p({l-n-1\choose r})=\d_p(n)/(p-1) -M$, so under the assumption that $\nu_p(l)\le \nu_p(n)$, we have $B_n^{(l-1)} (x)$ is a SAMPC
if and only if $B_n^{l)}(x)$ is an AMPC, and if $B_n^{(l-1)}(x)$ is a SAMPC or $B_n^{(l)}(x)$ is a AMPC then $\nu_p(B_n^{(l)})=\nu_p(B_n^{(l)} (1))$.
\item[(c)] If $\nu_p(l)> \nu_p(n)$ then $\nu_p(l+1)\le \nu_p(n)$, so we can apply (b) with $l$ replaced by $l+1$ and $M$ replaced by $M'$.
\end{itemize}
\end{thm}

The situation is slightly different for $p=2$, since there is a third canonical partition $\uu''$ if $n$ is odd, where $u_1'' =r-3=n-3$ and $u_3'' = 1$. The only difference is that instead of $\nu_p(\tau_\uu) =-M$
and $\nu_p(\tau_\uu)=-M'$, we now have $\nu_2(\tau_\uu)=-M$ or $\nu_2(\tau_{\uu''}) =-M$ (exclusive or) and $\nu_2(\tau_\uu)=-M'$ or $\nu_2(\tau_{\uu''})=-M'$ (exclusive or).  Clearly, $\nu_2(\tau_\uu) =-\d_2(n) +\nu_2({l-n-1\choose n})$ and $\nu_2(\tau_{\uu''})=1-\d_2(n) +\nu_2({l-n-1\choose n-2})$ for $B_n^{(l)}$.
The situation when $p=2$ will be made explicit in Theorem~\ref{th:5.x1-5.11}.\\

\begin{rem}
\label{rem:5.?}
The last assertion (parts (b) and (c) of the preceding theorem)  leads to Amdeberhan-type identities in several different contexts, e.g., for Stirling numbers of both kinds.
If (c) is true, replace $l$ by $l+1$.\\
\end{rem}

We have an analogous theorem for $p=2$, which follows from our analysis in \cite{Adseq}: there are now three critical partitions which can work, the partition $\uu$ where $u_1 =n$, the partition $\uu'$ where $u'_1=n-1$, and the partition $\uu''$ where $u_1''=n-3$ and $u_3 =1$. We have seen that  the partition $\uu'$ can only work if $\nu_2(l)\le \nu_2(n)$, and cannot work for $B_n^{(l)} (1)$ since $w(\uu')=n-1<n$.  This gives the following theorem.

\begin{thm}
\label{th:5.x1-5.11}
Let $p=2$, and let u be the  partition where $u_1 =n$, and let $\uu''$ be the partition where $u_1''=n-3$ and $u_3'' =1$. Let $M$ be the maximum pole of $B_n^{(l)} (x)$ and $M'$ be the maximum pole of $B_n^{(l+1)} (x)$. Then we have
 \begin{itemize}
 \item[(a)] $B_n^{(l-1)} (x)$ is a SAMZC if and only if $\nu_2(\tau_\uu)=-M=\nu_2(B_n^{(l)} (1))$ or $\nu_2(\tau_{\uu''})=-M=\nu_2(B_n^{(l)} (1))$ (exclusive or), where  $\nu_2(\tau_\uu)=-\d_2(n)+\nu_2({l-n-1\choose n})$ and $\nu_2(\tau_{\uu''})=1-\d_2(n)+\nu_2({l-n-1\choose n-2})$.
\item[(b)] If $\nu_2(l)\le \nu_2(n)$ then $B_n^{(l)}(x)$ is an AMPC if an only if $B_n^{(l-1)}(x)$ is a SAMPC, and if $B_n^{(l)} (x)$ is an AMPC or $B_n^{(l-1)}(x)$ is a SAMPC, then $\nu_2(B_n^{(l)})=\nu_2(B_n^{(l)} (1))=-M$.
\item[(c)] If $\nu_2(l)>\nu_2(n)$ then $\nu_2(l+1)=0$, so we can replace $l$ by $l+1$, and get $B_n^{(l+1)}(x)$ is an AMPC if and only if $B_n^{(l)}(x)$ is a SAMPC, and if $B_n^{(l+1)}(x)$ is an AMPC or $B_n^{(l)} (x)$ is a SAMPC then
$\nu_2(B_n^{(l+1)})=\nu_2(B_n^{(l+1)} (1))=-M'$.
\end{itemize}
\end{thm}


\begin{rem}
\label{rem:5.x1-5.11}
Each of these theorems translates to give a theorem about Stirling numbers of each kind.  In particular, the last equality of each theorem gives an Amdeberhan-type relation  in Section~\ref{sec:Amde} for Stirling numbers of the second kind and in Section~\ref{sec:main2.1} for Stirling numbers of the first kind.
\end{rem}


 For easy reference,  and for comparative purposes,  we include the analysis for $p=2$.  The criteria were given in \cite{Adseq} for Stirling numbers of the second kind, but not for Stirling numbers of the first kind.  Our analysis for the odd primes has enabled us to restate the criteria for $p=2$ in simpler form  (cf. \cite[Theorem~3.3]{Adseq}, which highlights the role of the condition $\nu_p(k)\le \nu_p(n)$ for $S(n,k)$.
Since it has not been previously done, we will give the AMZC and SAMZC criteria in our section on Stirling numbers of the first kind. \\

Specifically,
we note
that for $p=2$, the binomial coefficient ${a+b \choose a}$ has no unforced carries if and only if $\nu_2({a+b \choose a})=\#\{[a] \int[b]\}$.  Thus for $p=2$ and $B_{n-k}^{(-k)} (x)$, the three critical partitions $\uu$ are specified by  $u_1 =n-k=r$, and $u_1 =r-1$, and  $u_1=r-3$ and $u_3=1$, with $d=r$, $r-1$, or $r-2$ respectively. The criterion for AMZC
(cf. \cite[part (iii) of Theorem 3.3]{Adseq})
 can be restated as exactly one of the following conditions holds:
\begin{itemize}
\item[(i)] for the first partition $\uu$, there are no unforced carries for $n+d$;
\item[(ii)] for the second partition  $\uu$, there are no unforced carries for $n+d$ and $\nu_2(n)<\nu_2(k)$;
\item[(iii)] for the third partition $\uu$, there are no unforced carries for $n+d$,  and $k$ is odd and $n$ is even, and $\nu_2(r-1)=\nu_2(n)$.
\end{itemize}
Stating the conditions this way makes it very clear why (ii) and (iii) cannot both be true.\\

We are now ready to summarize and state more formally than we have before, the four estimates and cases for
Stirling numbers of the second kind for odd primes $p$.\\

\begin{defn}
\label{def:cases}
\newcommand{\nn}{n-k}
\phantom{xxx}
\begin{itemize}

\item[(1)] The minimum zero estimate is $\nu_p(S(n,k))\ge (\d_p(k)-\d_p(n))/(p-1)$.
When this estimate is sharp, we have a minimum zero case (MZC).  The estimate
was discovered by De Wannemacker in \cite{SdW}. His proof is unrelated to ours.
The MZC says the Newton polygon is strictly decreasing, which implies that $p-1 \mid \nn$,
and also that
$\nu_p(k)\le \nu_p(n)$.

\item[(2)] The shifted minimum zero estimate is $\nu_p(S(n,k))\ge (\d_p(k-1) -\d_p(n-1)/(p-1)$.
When this is sharp we have the shifted minimum zero case (SMZC).  The shift
improves the estimate if and only if
$\nu_p(n)<\nu_p(k)$.
The SMZC says the shifted higher order polynomial has a strictly decreasing Newton polygon,
again implying that $p-1 \mid \nn$.

The problem with these estimates is that they can be very crude, namely even
though $\nu_p(S(n,k))\ge 0$,
the estimates can be negative (vacuous). We can refine them be defining the almost minimum zero cases
(AMZC) and shifted almost minimum zero cases (SAMZC).
\item[(3)] {If $M$ is the maximum pole of $B_{n-k}^{(-k)} (x)$, then we have the almost minimum zero estimate 
\begin{linenomath}
\begin{align*}
\label{eq:AMZCp}
\nu_p(S(n,k))\ge \nu_p\left({n \choose k}\right)-M &=(\d_p(k)-\d_p(n))/(p-1)\\ &+(\d_p(n-k)/(p-1) -M).
\end{align*}
\end{linenomath}
We see that his estimate improves the minimum zero estimate in every case when that
estimate is not sharp.  It is not hard to show that this estimate is never vacuous (negative)
(cf. Theorem~\ref{th:2.x}).
When this estimate is sharp,
i.e., when $\nu_p(B_{n-k}^{(-k)})=-M$,
we have the AMZC. For simplicity,
we generally assume that $p-1 \mid \nn$.}
\item[(4)] {Finally, if $M'$ is the maximum pole of $B_{n-k}^{(-k+1)} (x)$, we have the shifted almost minimum zero
estimate,
\begin{linenomath}
\begin{align*}
\nu_p(S(n,k))&\ge \nu_p\left({n-1 \choose k-1}\right)-M' \\ &=(\d_p(k-1)-\d_p(n-1))/(p-1) +(\d_p(n-k)/(p-1) -M').
\end{align*}
\end{linenomath}
If this estimate is sharp,
i.e.,if $\nu_p(B_{n-k}^{(-k+1)}(1))=-M'$,
we have the SAMZC. Again for
simplicity, we generally assume that $p-1\mid \nn$.}
\end{itemize}
\end{defn}


\end{document}